\documentclass[a4paper,12pt]{article}
\usepackage[top=2cm, bottom=2cm, outer=2cm, inner=2cm, headsep=14pt]{geometry}
\usepackage{amsmath,amsfonts}
\usepackage{enumitem}
\usepackage{adjustbox}
\usepackage{blkarray}

\usepackage{multicol,color} 

\usepackage{tikz} 
\usetikzlibrary{decorations.pathreplacing,decorations.markings} 
\usepackage{url} 

\usepackage{soul,xcolor} 

\tikzset{
  on each segment/.style={
    decorate,
    decoration={
      show path construction,
      moveto code={},
      lineto code={
        \path [#1]
        (\tikzinputsegmentfirst) -- (\tikzinputsegmentlast);
      },
      curveto code={
        \path [#1] (\tikzinputsegmentfirst)
        .. controls
        (\tikzinputsegmentsupporta) and (\tikzinputsegmentsupportb)
        ..
        (\tikzinputsegmentlast);
      },
      closepath code={
        \path [#1]
        (\tikzinputsegmentfirst) -- (\tikzinputsegmentlast);
      },
    },
  },
  mid arrow/.style={postaction={decorate,decoration={
        markings,
        mark=at position .5 with {\arrow[#1]{stealth}}
      }}},
}

\usepackage{avant}

\allowdisplaybreaks

\newcommand{\G}{\Gamma}
\def\nnu{{\boldsymbol\nu}}
\newcommand{\Mat}{\mbox{\rm Mat}}

\def\O{{\boldsymbol O}}

\def\0{{\boldsymbol 0}}

\def\CC{{\mathbb C}}

\def\T{{\cal T}}

\def\E{{\cal E}}

\def\RR{{\mathbb R}}

\def\NN{{\mathbb N}}
\def\ZZ{{\mathbb Z}}
\def\A{{\mathcal A}}
\def\C{{\mathcal C}}

\def\R{{\mathcal R}}
\def\P{{\mathcal P}}
\def\B{{\mathcal B}}
\def\M{{\mathcal M}}
\def\D{{\mathcal D}}

\def\ol{\overline}
\def\ds{\displaystyle}

\def\Ra{\Rightarrow}
\def\la{\leftarrow}
\def\La{\Leftarrow}
\def\ra{\rightarrow}
\def\XXi{{\mathfrak X}}

\DeclareMathOperator{\Span}{span}
\DeclareMathOperator{\dgr}{dgr}

\DeclareMathOperator{\trace}{trace}

\DeclareMathOperator{\proj}{proj}

\DeclareMathOperator{\spec}{spec}

\DeclareMathOperator{\ii}{\boldsymbol{i}}
\DeclareMathOperator{\jj}{\boldsymbol{j}}
\DeclareMathOperator{\hh}{\boldsymbol{h}}
\DeclareMathOperator{\oo}{\boldsymbol{0}}

\DeclareMathOperator{\eig}{eig}

\newtheorem{theorem}{Theorem}[section]
\newtheorem{lemma}[theorem]{Lemma}
\newtheorem{corollary}[theorem]{Corollary}
\newtheorem{proposition}[theorem]{Proposition}
\newtheorem{definition}[theorem]{Definition}

\newtheorem{notation}[theorem]{Notation}
\newtheorem{remark}[theorem]{Remark}

\newtheorem{researchProblem}[theorem]{Research problem}

\newtheorem{problem}{Problem}[section]

\definecolor{ForestGreen}{RGB}{12, 110, 46}
\definecolor{ForestGreenTwo}{RGB}{120, 110, 86}

\newenvironment{proof}{{\noindent\it Proof. }}{\nopagebreak\hspace*{0.5cm}\hfill$\hbox{\rule{3pt}{6pt}}$\smallskip}

\definecolor{ForestGreen}{RGB}{12, 110, 46}
\definecolor{ForestGreenTwo}{RGB}{120, 110, 86}

\newfont\fiverm{cmr5}
\def\eeq{\end{equation}} 
\def\lbeq#1{\begin{equation} \label{#1}} 
 
\title{On combinatorial structure and algebraic characterizations of distance-regular digraphs}

\author{
{Giusy Monzillo}\\
{\small Faculty of Mathematics, Natural Sciences}\\
{\small and Information Technologies}\\
{\small University of Primorska}\\
{\small Muzejski trg 2, 6000 Koper, Slovenia }\\
{\small Giusy.Monzillo@famnit.upr.si} \and
{Safet Penji\'c}\\
{\small Faculty of Mathematics, Natural Sciences}\\
{\small and Information Technologies; and}\\
{\small Andrej Maru\v{s}i\v{c} Institute}\\
{\small University of Primorska}\\
{\small Muzejski trg 2, 6000 Koper, Slovenia }\\
{\small Safet.Penjic@iam.upr.si}
}

\begin{document}
\setstcolor{red}
\maketitle

\begin{abstract}
Let $\G=\G(A)$ denote a simple strongly connected digraph with vertex set $X$, diameter $D$, and let $\{A_0,A:=A_1,A_2,\ldots,A_D\}$ denote the set of distance-$i$ matrices of $\G$. Let $\{R_i\}_{i=0}^D$ denotes a partition of $X\times X$, where $R_i=\{(x,y)\in X\times X\mid (A_i)_{xy}=1\}$ $(0\le i\le D)$. The digraph $\G$ is distance-regular if and only if $(X,\{R_i\}_{i=0}^D)$ is a commutative association scheme. In this paper, we describe the combinatorial structure of $\G$ in the sense of equitable partition, and from it we derive several new algebraic characterizations of such a graph, including the spectral excess theorem for distance-regular digraph. Along the way, we also rediscover all well-known algebraic characterizations of such graphs.
\end{abstract}


\smallskip
{\small



}


\section{Introduction}
\label{1A}

Let $\G$ denote strongly connected digraph with diameter $D$ and adjacency matrix $A$. In this paper, we study connections between commutative association schemes and digraphs, by considering the following question: under which combinatorial restriction in the sense of equitable partitions on a digraph $\G=\G(A)$, the set $\{A^0,A^1,\ldots,A^D\}$ 
is a basis of the Bose--Mesner algebra of a commutative association scheme? Formal definitions are given in Section~\ref{2B}.

We first give the relevant background before presenting our main problem of research. Let $X$ denote a finite set, and $\Mat_X(\CC)$ the set of complex matrices with rows and columns indexed by $X$. Let $\R=\{R_0,R_1,\ldots,R_d\}$ denote a set of cardinality $d+1$ of nonempty subsets of $X\times X$. The elements of the set $\R$ are called {\em relations} (or {\em classes}) on $X$. For each integer $i$ $(0\le i\le d)$, let $B_i\in\Mat_X(\CC)$ denote the adjacency matrix of the graph $(X,R_i)$ (directed, in general). The pair ${\XXi}=(X,\R)$ is a {\em commutative $d$-class association scheme} if the {\em relation matrices} $B_i$'s satisfy the following properties
\begin{enumerate}[leftmargin=1.7cm,label=\rm(AS\arabic*)]
\item $B_0=I$, the identity matrix.\label{ce}
\item $\ds{\sum_{i=0}^d B_i=J}$, the all-ones matrix.\label{cg}
\item ${B_i}^\top\in\{B_0,B_1,\ldots,B_d\}$ for $0\le i\le d$.
\item $B_iB_j$ is a linear combination of $B_0,B_1,\ldots,B_d$ for $0\le i,j\le d$ (i.e., for every $i,j$ $(0\le i,j\le d)$ there exist positive integers $p^h_{ij}$ $(0\le h\le d)$, known as {\em intersection numbers}, such that $B_iB_j=\sum_{h=0}^d p^h_{ij} B_h$).\label{cb}
\item $B_iB_j=B_jB_i$ for every $i,j$ $(0\le i,j\le d)$ (i.e., for the intersection numbers $p^h_{ij}$, $0\le i,j,h\le d$, from \ref{cb} we have that $p^h_{ij}=p^h_{ji}$).\label{cf}
\end{enumerate}

By \ref{ce}--\ref{cf} the vector space $\M=\Span\{B_0,B_1,\ldots,B_d\}$ is a commutative algebra; it is known as the {\em Bose--Mesner algebra} of $\XXi$.  We say that a matrix $B$ {\em generates} $\M$ if every element in $\M$ can be written as a polynomial in $B$. We say that a (di)graph $\G=\G(A)$ {\em generates} a commutative association scheme $\XXi$ if and only if the adjacency matrix $A$ of $\G$ generates the Bose--Mesner algebra $\M$ of $\XXi$, and in symbols we write $\M=(\langle A\rangle,+,\cdot)$. We say that $\XXi$ is {\em symmetric} if the $B_i$'s are symmetric matrices.

It is well known that an undirected graph $\G$ of diameter $D$ is a distance-regular graph if and only if $\G$ generates a symmetric $D$-class association scheme (see, for example, \cite{BCN,FaD,DKT,SPm}). Moreover, an undirected graph $\G$ with $d+1$ distinct eigenvalues is a quotient-polynomial graph if and only if $\G$ generates a symmetric $d$-class association scheme (see \cite{FQpG,FMPS}). Some combinatorial and algebraic properties of a digraph $\G$ with $d+1$ distinct eigenvalues which generates a non-symmetric $d$-class association scheme are described in \cite{MPc}. In this paper, we are interested in the case when a digraph $\G$ with diameter $D$ generates a commutative $D$-class association scheme. This problem can be set up in a completely different way, as follows:

\begin{problem}{\label{Aa}}{\rm
Let $\G$ denote a simple strongly connected digraph with diameter $D$, and let $\{A_0,A:=A_1,\ldots,A_D\}$ denote the distance-$i$ matrices of $\G$. Consider the following two properties of $\G$.
\begin{enumerate}[label=\rm(\roman*)]
\item $A^\top\in\{A_0,A_1,\ldots,A_D\}$.
\item There exist complex scalars $p^h_{i1}$ $(0\le i,h\le D)$ such that 
$$
A_iA=\sum_{h=0}^D p^h_{i1} A_h
\qquad(0\le i\le D).
$$
\end{enumerate}
Using the algebraic properties (i) and (ii) from above, can we describe the combinatorial structure of $\G$ in the sense of equitable partitions? The vice-versa question is also of importance, i.e., what combinatorial structure does a digraph need to have so that the property (i) and (ii) from the above hold?
}\end{problem}

Note similarity and difference between properties (i), (ii) of Problem~\ref{Aa} and (AS3), (AS4) from definition of $D$-class association schemes. In this paper, we show that these two properties hold if and only if $\G$ generates a commutative $D$-class association scheme (see Section~\ref{eA}).

Let $\G=\G(A)$ denote a simple strongly connected digraph with vertex set $X$, diameter $D$, and let $\{A_0,A:=A_1,A_2,\ldots,A_D\}$ denote the distance-$i$ matrices of $\G$. Pick $x\in X$, and let $\G^\ra_i(x):=\{z\in X\mid (A_i)_{xz}=1\}$ and $\G^\la_i(x):=\{z\in X\mid (A_i)_{zx}=1\}$ $(0\le i\le D)$. We denote by $\partial(x,y)$ the distance from vertex $x$ to vertex $y$. Since we are considering oriented graphs, notice that $\partial(x,y)$ and $\partial(y,x)$ may not be equal. In this paper, we define distance-regular digraphs in the following way: A digraph $\G$ is {\em distance-regular} if for every vertex $x\in X$ the distance-$i$ partition 
$$
\{\G^\ra_0(x),\G^\ra_1(x),\ldots,\G^\ra_D(x)\}
$$
of $X$ is equitable and its corresponding parameters do not depend on the choice of the vertex $x$. The algebra $\A$ generated by the adjacency matrix $A$ is called the {\em adjacency algebra} of $\G$. Our main result is Theorem~\ref{Ab}.

\begin{theorem}
\label{Ab}
Let $\G=\G(A)$ denote a simple strongly connected digraph with vertex set $X$, diameter $D$, adjacency algebra $\A$, and let $\{A_0,A:=A_1,A_2,\ldots,A_D\}$ denote the distance-$i$ matrices of $\G$. Then, the following are equivalent.
\begin{enumerate}[label=\rm(\roman*)]
\item $\G$ is a distance-regular digraph (i.e., for every vertex $x\in X$, the distance-$i$ partition $\{\G^\ra_0(x),\G^\ra_1(x),\ldots,\G^\ra_D(x)\}$ of $X$ is equitable and its corresponding parameters do not depend on the choice of $x$).\label{Xh}
\item For every vertex $x\in X$, the partition $\{\G^\la_0(x),\G^\la_1(x),\ldots,\G^\la_D(x)\}$ of $X$ is equitable and its corresponding parameters do not depend on the choice of $x$.
\item $(X,\{R_i\}_{i=0}^D)$ is a commutative association scheme, where $\{R_i\}_{i=0}^D$ denotes a partition of $X\times X$ with $R_i=\{(x,y)\in X\times X\mid (A_i)_{xy}=1\}$ $(0\le i\le D)$. \label{Xe}
\item $\A$ is the Bose--Mesner algebra of a commutative $D$-class association scheme.\label{Xm}
\item $A^\top\in\{A_0,A_1,\ldots,A_D\}$, and there exist complex scalars $p^h_{i1}$ $(0\le i,h\le D)$ such that 
$$
A_iA=\sum_{h=0}^D p^h_{i1} A_h
\qquad(0\le i\le D).
$$
\item $(X,\{R_{\ii}\}_{\ii\in\Delta})$ is a commutative $D$-class association scheme, where $\Delta=\{(\partial(x,y),\partial(y,x))\mid x,y\in X\}$, and for any $\ii\in\Delta$, $R_{\ii}$ denote the set of ordered pairs $(x,y)\in X\times X$ such that $(\partial(x,y),\partial(y,x))=\ii$.\label{Xi}
\item $\G$ is a regular digraph with $d + 1$ distinct eigenvalues, has spectrally maximum diameter $D = d$, and the matrices $A^\top$ and $A_D$ are polynomials in $A$.
\item $\G$ is a regular digraph with $d + 1$ distinct eigenvalues, has spectrally maximum diameter $D = d$, $A^\top$ is a polynomial in $A$, and the following equality holds 
$$
\frac{1}{|X|} \sum_{x\in X} \G^{\ra}_D(x)=|X|\left(\sum_{j=0}^d \frac{|\pi_0|^2}{m(\lambda_j)|\pi_j|^2}\right)^{-1},
$$
where $\spec(\G)=\{[\lambda_0]^{m(\lambda_0)},\ldots,[\lambda_d]^{m(\lambda_d)}\}$ is the spectrum of $\G$ and $\pi_i=\prod_{j=0\,(j\not=i)}^d(\lambda_i-\lambda_j)$ $(0\le i\le d)$ (this is a {\bf spectral excess theorem} for distance-regular digraphs).
\label{Xn}
\end{enumerate}
\end{theorem}

The concept of a distance-regular digraph was introduce by {\sc Damerell} \cite{Drm} in the late $1970$s. A digraph $\G$ 
is distance-regular in sense of {\sc Damerell} \cite{Drm} if the following property holds:
\begin{enumerate}[left=1cm]
\item[({\sc Dam})] There exist numbers $b_{ij}$ $(0\le i,j\le D)$ such that $|\G^\ra_1(y)\cap\G_j^\ra(x)|=b_{ij}$, for all $x\in X$, $y\in\G^\ra_i(x)$ $(0\le i,j\le D)$.\label{Xd}
\end{enumerate}
In \cite[Theorem~3]{Drm} {\sc Damerell} proved that ({\sc Dam}) yields property \ref{Xe} of Theorem~\ref{Ab}. An interesting question is whether the property \ref{Xe} of Theorem~\ref{Ab} yields ({\sc Dam}). We give the answer to this question. Moreover, in Section~\ref{iA}, we prove that {\sc Damerell} definition ({\sc Dam}) of distance-regular digraph is equivalent to our definition of distance-regular digraph (see Theorem~\ref{Ab}(i)). In the same paper \cite{Drm}, {\sc Damerell} had extended some results of {\sc Bannai et al.} (\cite{BCK}) and {\sc Lam} (\cite{Lam}) from distance-transitive digraphs to the case of distance-regular digraphs (for example, if $g\ge 3$ is the girth of a  distance-regular digraph $\G$, then $\G$ has diameter $D=g$ (`long' distance-regular digraph) or $D=g-1$ (`short' distance-regular digraph); case $g=2$ corresponds to case of distance-regular graphs). For $g\in\{4,5,6\}$ there are some works where, among else, the intersection numbers are computed by using the small number of defining parameters (\cite{EM,LM,Tt,Ym}).

{\sc Douglas} and {\sc Nomura} in \cite{LN} defined distance-regular digraph in the following way:
\begin{enumerate}[left=2cm]
\item[({\sc DouNom})] A strongly connected digraph (without loops) with vertex set $X$ such that the size of $s^h_{ij}:=|\{x\in X\mid \partial(u,x)=i, \partial(v,x)=j \}|$ depends only on $i$, $j$ and $\partial(u,v)$, rather than the individual vertices $u$, $v$ with $\partial(u,v)=h$ (\cite[pages~34,~35]{LN}).
\end{enumerate}
The authors showed that the girth $g$ of a `short' distance-regular digraph with degree $>1$ and diameter $D$ always satisfies $D\le g\le 8$. A distance regular-digraph with diameter $D=3$ generates a nonsymmetric $3$-class association scheme. In \cite{MPc}, {\sc Monzillo} and {\sc Penjić} showed that, if $\XXi$ is a commutative $3$-class association scheme that is not an amorphic symmetric scheme, then we can always find a (di)graph $\G$ which generates $\XXi$; from Theorem~\ref{Ab} it follows that, if the diameter of such a graph is equal to $d$ (where $d+1$ is the number of distinct eigenvalues), then $\G$ is distance-regular. In this light it would be interesting to  carefully investigate when $3$-class scheme contain a strongly connected graph of diameter $3$ as well as a correspondence between the case $g=3$ of distance-regular digraphs and the so-called {\em skew Hadamard matrices} (see, for example, \cite{GCf,GCb,GC,IM,Jl}). {\sc Leonard} in \cite{Lda} defined distance-regular digraphs as a commutative $P$-polynomial association schemes $\XXi$, i.e., he defined distance-regular digraph as the following association scheme:
\begin{enumerate}[left=1cm]
\item[({\sc Leo})] $(X,\{R_i\}_{i=0}^D)$ is a commutative association scheme with adjacency matrices $\{B_i\}_{i=0}^d$ such that for some ordering $(B_0,B:=B_1,\ldots,B_d)$ of the adjacency matrices, there exist polynomials $p_i$ $(0\le i\le d)$ with $\dgr p_i\le i$ and $B_i=p_i(B)$.
\end{enumerate}
In Section~\ref{dA}, we show (re-discover) that any $P$-polynomial commutative (nonsymmetric) association scheme with a $P$-polynomial ordering $(B_0,B=B_1,\ldots,B_d)$ produces distance-regular digraph $\G=\G(B)$. Using ({\sc Leo}) as definition of distance-regular digraph, {\sc Leonard} in \cite{Lda} studied $Q$-polynomial distance-regular digraphs. In \cite{Ma}, {\sc Munemasa} used Theorem~\ref{Ab}\ref{Xe} as definition of distance-regular digraph in order to study nonsymmetric $P$- and $Q$-polynomial association schemes. Some other results about distance-regular digraphs can be found, for example, in \cite{AOST,CT,Lda}.

A digraph $\G$ is weakly distance-regular in sense of {\sc Comellas et al.} \cite{CFGM} if the number of walks of length $\ell$ $(0\le\ell\le D)$ in $\G$ between two vertices $x,y\in X$ only depends on $h=\partial(x,y)$ (on the distance from $x$ to $y$). In \cite[Theorem~2.2]{CFGM} {\sc Comellas et al.} proved that the following {\rm (Ci)--(Civ)} are equivalent:
(Ci)~$\G$ is weakly distance-regular in sense of {\sc Comellas et al.} \cite{CFGM}; 
(Cii)~The distance-$i$ matrix $A_i$ is a polynomial of degree $i$ in the adjacency matrix $A$; 
(Ciii)~The set of distance-$i$ matrices $\{A_i\}_{i=0}^D$ is a basis of the adjacency algebra $\A$ of $\G$; 
(Civ)~For any two vertices $u,v\in X$ at distance $\partial(u,v)=h$, the numbers $p^h_{ij}(u,v):=|\G_i^\ra(u)\cap\G_j^\la(v)|$ do not depend on the vertices $u$ and $v$, but only on their distance $h$.
For the moment consider the following property of a digraph $\G=\G(A)$:
\begin{enumerate}[left=1cm]
\item[({\sc Com})] $A$ is a normal matrix and $\G$ is a weakly distance-regular in sense of {\sc Comellas et al.}~\cite{CFGM}.
\end{enumerate}
In \cite[Proposition~2.6]{CFGM} {\sc Comellas et al.} proved an equivalence between the properties ({\sc Dam}) and ({\sc Com}). Among else, they also studied algebraic-combinatorial characterizations of such graphs (see \cite[Theorem~2.2]{CFGM}) and the spectrum of weakly distance-regular digraphs, as well as some constructions of such graphs. Some papers that are related with combinatorial properties of weakly distance-regular digraphs (in the sense of number of walks) are \cite{CfFgM,GzJdZ,LL,DO}. For the study of the spectrum of digraphs related with weakly distance-regular property in the sense of {\sc Comellas et al.} \cite{CFGM}, we recommend papers \cite{FaMm,Ogr}.

For the moment let $\G$ denote digraph with vertex set $X$, and let $\Delta_{(x,y)}=(\partial(x,y),\partial(y,x))$ denote an ordered pair in which the first coordinate is the distance from $x$ to $y$, and the second coordinate is the distance from $y$ to $x$. The ordered pair $\Delta_{(x,y)}$ is called {\em two-way distance} between $x$ and $y$. Let 
$$
\Delta=\{\Delta_{(x,y)}\mid x,y\in X\}=\{(\partial(x,y),\partial(y,x))\mid x,y\in X\}
$$ 
denote the set of all (pairs of) two-way distances $\Delta_{(x,y)}$. For any $\ii\in\Delta$, let $R_{\ii}$ denote the set of ordered pairs $(x,y)$ with $(\partial(x,y),\partial(y,x))=\ii$. Note that $\{R_{\ii}\}_{\ii\in\Delta}$ is a partition of $X\times X$. A digraph $\G$ is said to be weakly distance-regular in sense of {\sc Wang} and {\sc Suzuki} \cite{WS} if, for any $\hh,\ii,\jj\in\Delta$ and $\Delta_{(x,y)}=\hh$, the number of $z\in X$ such that $\Delta_{(x,z)}=\ii$ and $\Delta_{(z,y)}=\jj$ depends only on $\hh$, $\ii$ and $\jj$. To explain this notation in sense of association schemes, for the moment let $\0=(0,0)$. Note that (AS1') $R_{\oo}=\{(x,x)\mid x\in X\}$; and (AS2') $\{R_{\ii}\}_{\ii\in\Delta}$ is a partition of the Cartesian product $X\times X$. We also have (AS3') the relation $R^\top_{\jj}=\{(y,x)\mid (x,y)\in R_{\jj}\}$ is in $\{R_{\ii}\}_{\ii\in\Delta}$; as well as (AS4'), for each triple $\ii,\jj,\hh$ $(\ii,\jj,\hh\in\Delta)$ and $(x,y)\in R_{\hh}$, the scalar
$$
|\{z\in X \mid (x,z)\in R_{\ii} \mbox{ and } (z,y)\in R_{\jj} \}|
$$
does not depend on the choice of the pair $(x,y)\in R_{\hh}$. This yields that  a digraph $\G$ is weakly distance-regular in sense of {\sc Wang} and {\sc Suzuki} \cite{WS} if and only if $(X,\{R_{\ii}\}_{\ii\in\Delta})$ is a $|\Delta|$-class association scheme (for a different approach in the explanation of this fact, see {\sc Suzuki} \cite{ShT}). In such a case, $\XXi(\G)$ is called the {\em attached scheme} of $\G$. 
In Theorem~\ref{Ab} (see property~\ref{Xi}), we show which subfamily of weakly distance-regular digraphs in the sense of {\sc Wang} and {\sc Suzuki} \cite{WS} are distance-regular. 
In \cite{MPc} we have studied the case when $|\Delta|=d+1$, where $d+1$ is the number of distinct eigenvalues of $\G$. In \cite{ZYW} {\sc Zeng}, {\sc Yang} and {\sc Wang} characterized all distance-regular digraphs in sense of isomorphisms between distance-regular digraphs and the relations $R_i$ of an attached association scheme of $\G$. In \cite{LGG,ShT,Wk,Yi,Yii} some special families of weakly distance-regular digraphs in sense of {\sc Wang} and {\sc Suzuki} of small valency have been classified. Algebraic restrictions on weakly distance-regular digraphs, called thin, quasi-thin and thick, are studied in \cite{ShT,Yqt,YWt}. Other papers that are directly or indirectly involved in the study of weakly distance-regular digraphs in sense of {\sc Wang} and {\sc Suzuki} are, for example, \cite{FW,FZL,GK,LS,Wki,WF}.

The contribution of this paper is in a combinatorial meaning of distance-regular digraph in sense of equitable partition, from which we get several new characterizations. As we have mentioned above, under the assumption that $A$ is a normal matrix, a combinatorial meaning of distance-regular digraph in sense of number of walks is already known from \cite{CFGM}. In this paper, we start with pure a combinatorial definition (from the equitable partition, see property~\ref{Xh} of Theorem~\ref{Ab}) and we proceed step by step to derivate all other properties (all of them have a similar counterpart in the theory of undirected distance-regular graphs). In the way, from our definition we also rediscover all known results from theory of distance-regular digraphs. As far as we know, no one gave some combinatorial meaning of distance-regular digraphs in sense of equitable partitions. The algebraic properties~\ref{Xm}--\ref{Xn} of Theorem~\ref{Ab} (the last property is the spectral excess theorem for distance-regular digraphs) are of  particular interest. Our proof of Theorem~\ref{Ab} use linear algebra technique, combinatorial technique, as well ass ``usual'' trick that can be found, for example in, \cite{CFGM,FaD,GRac,LpTm,SPm,SP,Rp}. To chase our results, we got inspiration from two {\sc Fiol}'s papers \cite{FaD,FQpG}.

Our paper is organized as follows. In Section~\ref{2B}, we recall basic concepts from algebraic graph theory, including commutative association schemes (experts from the field can skip this section). Section~\ref{Av} is an overview of all properties that we use later in the paper: in this section, we show what for a graph means if $A^\top,A_iA \in\{A_0,A_1,\ldots,A_D\}$ $(0\le i\le D)$, where $A_i$'s the are distance-$i$ matrices of the graph. Our paper then starts from Section~\ref{Ga}, where we define distance-regular digraphs in terms of equitable partitions. Through Sections~\ref{dA}--\ref{qa}, we reprove all-well known algebraic characterizations, and provide new ones. We finish the paper with Section~\ref{jD}, where we describe possible further directions. {\color{blue} This manuscript has $37$ pages, but actually the whole paper has $20$ pages (see Section~\ref{Ga}--\ref{qa}).}

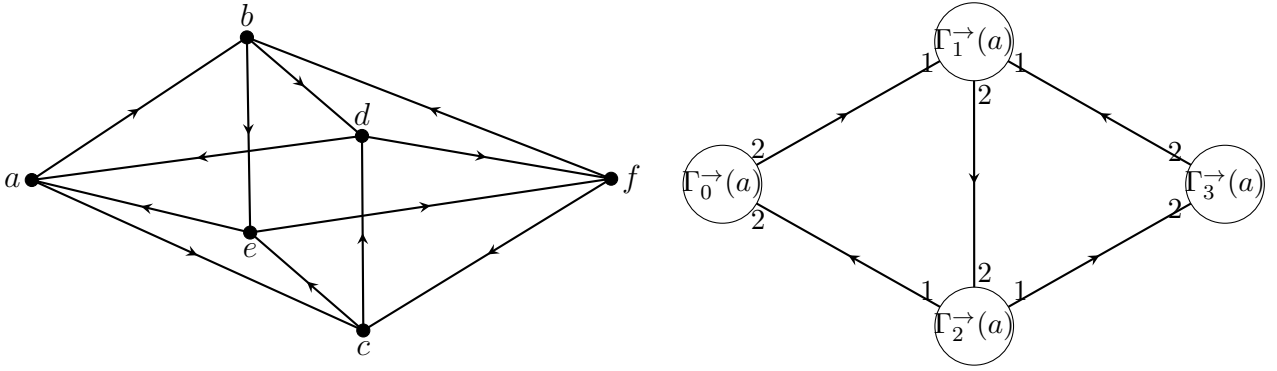
\begin{figure}[t!]
\begin{center}
\begin{tikzpicture}[scale=.43]
\draw [line width=.8pt,draw=black,postaction={on each segment={mid arrow=black}}] (-9.72,2.6) -- (-3.14,7);
\draw [line width=.8pt,draw=black,postaction={on each segment={mid arrow=black}}] (-9.72,2.6) -- (0.42,-2.04);
\draw [line width=.8pt,draw=black,postaction={on each segment={mid arrow=black}}] (0.38,3.96) -- (-9.72,2.6);
\draw [line width=.8pt,draw=black,postaction={on each segment={mid arrow=black}}] (-3.04,0.98) -- (-9.72,2.6);
\draw [line width=.8pt,draw=black,postaction={on each segment={mid arrow=black}}] (-3.14,7) -- (-3.04,0.98);
\draw [line width=.8pt,draw=black,postaction={on each segment={mid arrow=black}}] (0.42,-2.04) -- (0.38,3.96);
\draw [line width=.8pt,draw=black,postaction={on each segment={mid arrow=black}}] (8.0001,2.64) -- (-3.14,7);
\draw [line width=.8pt,draw=black,postaction={on each segment={mid arrow=black}}] (0.38,3.96) -- (8.0001,2.64);
\draw [line width=.8pt,draw=black,postaction={on each segment={mid arrow=black}}] (-3.04,0.98) -- (8.0001,2.64);
\draw [line width=.8pt,draw=black,postaction={on each segment={mid arrow=black}}] (8.0001,2.64) -- (0.42,-2.04);
\draw [line width=.8pt,draw=black,postaction={on each segment={mid arrow=black}}] (0.42,-2.04) -- (-3.04,0.98);
\draw [line width=.8pt,draw=black,postaction={on each segment={mid arrow=black}}] (-3.14,7) -- (0.38,3.96);
\draw [fill=black] (-9.72,2.6) circle [radius=0.2];
\node at (-9.72,2.6) [anchor=east] {$a$};
\draw [fill=black] (-3.14,7) circle [radius=0.2];
\node at (-3.14,7) [anchor=south] {$b$};
\draw [fill=black] (0.42,-2.04) circle [radius=0.2];
\node at (0.42,-2.04) [anchor=north] {$c$};
\draw [fill=black] (0.38,3.96) circle [radius=0.2];
\node at (0.38,3.96) [anchor=south] {$d$};
\draw [fill=black] (-3.04,0.98) circle [radius=0.2];
\node at (-3.04,0.98) [anchor=north] {$e$};
\draw [fill=black] (8.0001,2.64) circle [radius=0.2];
\node at (8.0001,2.64) [anchor=west] {$f$};
\end{tikzpicture}~
{\small
\begin{tikzpicture}[scale=.47]
\draw[line width=.8pt,draw=black,postaction={on each segment={mid arrow=black}}] (-7.04,-0.01) -- (0,4.01);
\draw[line width=.8pt,draw=black,postaction={on each segment={mid arrow=black}}] (0.006,-4.014) -- (-7.04,-0.01);
\draw[line width=.8pt,draw=black,postaction={on each segment={mid arrow=black}}] (0,4.01) -- (0.006,-4.014);
\draw[line width=.8pt,draw=black,postaction={on each segment={mid arrow=black}}] (7.024,-0.01) -- (0,4.01);
\draw[line width=.8pt,draw=black,postaction={on each segment={mid arrow=black}}] (0.006,-4.014) -- (7.024,-0.01);
\draw [fill=white] (-7.04,-0.01) circle [radius=1.1];
\draw [fill=white] (0,4.01) circle [radius=1.1];
\draw [fill=white] (0.006,-4.014) circle [radius=1.1];
\draw [fill=white] (7.024,-0.01) circle [radius=1.1];
\node at (-7.04,-0.01) {\small$\G^\ra_0(a)$};
\node at (0,4.01) {\small$\G^\ra_1(a)$};
\node at (0.006,-4.014) {\small$\G^\ra_2(a)$};
\node at (7.024,-0.01) {\small$\G^\ra_3(a)$};
\node  at (-6.04,0.99) {$2$};
\node  at (-6.04,-1.01) {$2$};
\node  at (1.3,3.41) {$1$};
\node  at (-1.3,3.41) {$1$};
\node  at (0.3,2.51) {$2$};
\node  at (1.3,-3) {$1$};
\node  at (-1.3,-3) {$1$};
\node  at (0.3,-2.51) {$2$};
\node  at (5.624,0.9) {$2$};
\node  at (5.624,-0.7) {$2$};
\end{tikzpicture}
}
\caption{Example of a distance-regular digraph and its distance-$i$ diagram around the vertex~$a$. Note that $\G_1^{\ra}(a)=\{b,c\}$, $\G_2^{\ra}(a)=\{d,e\}$ and $\G_3^{\ra}(a)=\{f\}$.}
\label{2f}
\end{center}
\end{figure}


\section{Preliminaries}
\label{2B}

{\color{blue}
{\bf Comment that we will delete from the final version of the paper:} 
Sections~\ref{2B} and \ref{Av} have $10$ pages. We expect that those readers who are familiar with the field will probably skip these sections and start to read the paper from Section~\ref{Ga}. With a shorter version of these two sections, we found that the paper is not readable as we want it to be (for all readers). 
}

A {\em digraph} with {\em vertex set} $X$ and {\em arc set} $\E$ is a pair $\G=(X, \E)$ which consists of a finite set $X=X(\G)$ of {\em vertices} and a set $\E = \E(\G)$ of {\em arcs} ({\em directed edges}) between vertices of $\G$. As the initial and final vertices of an arc are not necessarily different, the digraphs may have loops (arcs from a vertex to itself), and multiple arcs, that is, there can be more than one arc from each vertex to any other. If $e = (x,y)\in\E$ is an arc from $x$ to $y$, then the vertex $x$ (and the arc $e$) is {\em adjacent to} the vertex $y$, and the vertex $y$ (and the arc $e$) is {\em adjacent from} $x$. The {\em converse directed graph} $\ol{\G}$ is obtained from $\G$ by reversing the direction of each arc. 
For a vertex $x$, let $\G_1^{\la}(x)$ (resp. $\G_1^{\ra}(x)$) denote the set of vertices adjacent to (resp. from) the vertex $x$. In another words,
$$
\G_1^{\ra}(x) = \{z\mid (x,z)\in \E(\G)\}
\qquad\mbox{ and }\qquad
\G_1^{\la}(x) = \{z\mid (z,x)\in \E(\G)\}.
$$
Two small comments about the above notation: (i) drawing directed edge from $x$ to $z$, we have $x\ra z$, which yields the idea of using the notation $\G_1^{\ra}(x)$; (ii) drawing directed edge from $z$ to $x$, we have $x\la z$ (or $z\ra x$), which yields the idea of using the notation $\G_1^{\la}(x)$. The elements of $\G_1^{\ra}(x)$ are called {\em neighbors} of $x$. Instead of a set of vertices, we can consider a set of arcs: for a vertex $y$, let $\delta_1^{\la}(y)$ (resp. $\delta_1^{\ra}(y)$) denote the set of arcs adjacent to (resp. from) the vertex $y$. The number $|\delta_1^{\ra}(y)|$ is called the {\em out-degree of $y$} and is equal to the number of edges leaving $y$. The number $|\delta_1^{\la}(y)|$ is called the {\em in-degree of $y$} and is equal to the number of edges going to $y$. A digraph $\G$ is {\em $k$-regular} (of valency $k$) if $|\delta^{\ra}_1(y)| = |\delta_1^{\la}(y)| = k$ for all $y\in X$. We call $\G$ {\em simple}, if $\G$ contains neither loops nor multiple edges. In this paper we will always assume that our digraph is simple and strongly connected.

Let $\G = (X,\E)$ denote a digraph. For any two vertices $x, y \in X$, a {\em directed walk} of length $h$ from $x$ to $y$ is a sequence $[x_0,x_1,x_2,\ldots,x_h]$ $(x_i\in X,\, 0\le i\le h)$ such that $x_0 = x$, $x_h = y$, and $x_i$ is adjacent to $x_{i+1}$ (i.e. $x_{i+1}\in\G^{\ra}_1(x_i)$) for $0\le i\le h-1$. We say that $\G$ is {\em strongly connected} if for any $x, y\in X$ there is a directed walk from $x$ to $y$. A {\em closed directed walk} is a directed walk from a vertex to itself. A {\em directed path} is a directed walk such that all vertices of the directed walk are distinct. A {\em cycle} is a closed directed path. The {\em girth} of $\G$ is the length of a shortest cycle in $\G$.

For any $x, y\in X$, the {\em distance} from $x$ to $y$ (or between $x$ and $y$), denoted by $\partial(x, y)$, is the length of a shortest directed path from $x$ to $y$. The {\em diameter} $D = D(\G)$ of a strongly connected digraph $\G$ is defined to be 
$$
D = \max\{\partial(y,z)\,|\,y, z\in X\}.
$$
For a vertex $x\in X$ and any non-negative integer $i$ not exceeding $D$, let $\G^\ra_i(x)$ (or $\G_i(x)$) denote the subset of vertices in $X$ that are at distance $i$ from $x$, i.e.,
$$
\G^\ra_i(x)=\{z\in X\mid \partial(x,z)=i\}.
$$
We also define the set $\G^\la_i(x)$ as $\G^\la_i(x)=\{z\in X\mid \partial(z,x)=i\}$. Let $\G_{-1}(x) = \G_{D+1}(x) := \emptyset$. The {\it eccentricity} of $x$, denoted by $\varepsilon=\varepsilon(x)$, is the maximum distance between $x$ and any other vertex of $\G$. Note that the diameter of $\G$ equals $\max\{\varepsilon(x)\mid x\in X\}$.

All undirected graphs in this paper can be understood as digraphs in which an undirected edge between two vertices $x$ and $y$ represents two arcs, an arc from $x$ to $y$, and an arc from $y$ to $x$. In diagrams instead of drawing two arcs we draw one undirected edge between vertices $x$ and $y$.  For a basic introduction to the theory of undirected graphs we refer to \cite[Section~2]{FMPS}. With the word {\em graph} we refer to a finite simple directed or undirected graph.

\subsection{Equitable and distance-faithful partition}

A {\it partition} of a digraph $\G$ is a collection $\{\P_0, \P_1, \dots, \P_s\}$ of nonempty subsets of the vertex set $X$, such that $X=\ds\bigcup_{i=0}^s \P_i$ and $\P_i\cap\P_j=\emptyset$ for all $i,j$ $(0 \le i,j \le s,\, i\ne j)$. An {\it equitable partition} of a digraph $\G$ is a partition $\{\P_0, \P_1, \dots, \P_s\}$ of its vertex set, such that for all integers $i,j$ $(0 \le i,j \le s)$ the following two conditions hold.
\begin{enumerate}[label=\rm(\roman*)]
\item The number $d^{\ra}_{ij}$ of neighbors which a vertex in the cell $\P_i$ has in the cell $\P_j$ is independent of the choice of the vertex in $\P_i$ (i.e., for every $y\in\P_i$ we have $|\G^{\ra}_1(y)\cap\P_j|=d^{\ra}_{ij}$).
\item The number $d^{\la}_{ij}$ of vertices from the cell $\P_j$ which are adjacent to a vertex in $\P_i$ is independent of the choice of the vertex in $\P_i$ (i.e., for every $y\in\P_i$ we have $|\sum_{z\in\P_j}|\G^{\ra}_1(z)\cap\{y\}|=d^{\la}_{ij}=|\G_1^\la(y)\cap\P_j|$).
\end{enumerate}
We call the numbers $d^{\ra}_{ij}$, $d^{\la}_{ij}$ $(0\le i,j\le s)$ the {\em corresponding parameters}.

A {\it distance partition around} $x$ of a digraph $\G$ with vertex set $X$ is a partition $\{\G^{\ra}_0(x)=\{x\},\G^{\ra}(x),\ldots,\G^{\ra}_{\varepsilon(x)}(x)\}$ of $X$ where $\varepsilon(x)$ is eccentricity of $x$. An {\it $x$-distance-faithful partition} $\{\P_0=\{x\},\P_1,\ldots,\P_s\}$ (with $s\ge\varepsilon$) is a refinement of the distance partition around $x$ (here refinement means that some of $\G_i(x)$ can be equal to a union of some $\P_h$'s).

The {\it distance-$i$ diagram} of an equitable partition $\Pi=\{\G^{\ra}_0(x)=\{x\},\G^{\ra}(x),\ldots,\G^{\ra}_{\varepsilon(x)}(x)\}$ (around $x$) of a graph $\G$ is a collection of circles indexed by the sets of $\Pi$ with directed edges between some of them. If there is no directed edge between $\G^\ra_i(x)$ and $\G^\ra_j(x)$, then it means that there is no directed edge $yz$ for any $y\in\G^\ra_i(x)$ and $z\in\G^\ra_j(x)$. If there is a directed edge between $\G^\ra_i(x)$ and $\G^\ra_j(x)$, then a number on the line (from $\G^\ra_i(x)$ to $\G^\ra_j(x)$) near the circle $\G^\ra_i(x)$ denotes the corresponding parameter $d^{\ra}_{ij}$. A number above or below a circle $\G^\ra_i(x)$ denotes the corresponding parameter $d^{\ra}_{ii}(=d^{\la}_{ii})$. A similar explanation holds for the corresponding parameter $d^{\la}_{ij}$ (see Figures~\ref{2f} for an example).

We say that the combinatorial structure of the distance-$i$ diagram is {\em the same} around every vertex if for every vertex $x$ there exists an $x$-distance-faithful equitable partition with $D$ cells of same cardinalities and (same) corresponding parameters do not depend on the choice of~$x$.


\subsection{Elementary algebraic graph theory and the standard basis}

In this section, we recall some definitions and basic concepts from algebraic graph theory.

The {\em adjacency matrix} $A\in\Mat_X(\CC)$ of a digraph $\G$ (with vertex set $X$) is indexed by the vertices from $X$, and is defined in the following way
$$
\mbox{$(A)_{yz} =$ the number of arcs from $y$ to $z$}\qquad (y,z\in X)
$$
(note that $(A)_{yz}\in\ZZ^+_0$). Moreover, if we allow loops, the diagonal entries of $A$ can be different from zero. If $\G$ is simple graph (i.e., if the adjacency matrix $A$ is a $01$-matrix), then the $yz$-entry of the power $A^\ell$ $(\ell\in\NN)$ corresponds to the number of $\ell$-walks from the vertex $y$ to the vertex $z$ in $\G$ (see Lemma~\ref{hB}). As we already mention, in this paper we will always work with simple and strongly connected graphs.

\begin{lemma}
\label{hB}
Let $\G$ denote a simple strongly connected digraph with vertex set $X$, diameter $D$ and adjacency matrix $A$. The number of walks of length $\ell\in\NN$ in $\G$ from $x$ to $y$ is equal to $(x,y)$-entry of the matrix $A^\ell$.
\end{lemma}

\begin{proof}
Routine.
\end{proof}

The distance-$i$ matrix $A_i$ of a digraph $\G$ with diameter $D$ and vertex set $X$ is defined by
$$
(A_i)_{zy}=\left\{\begin{matrix}
1&\mbox{ if $\partial(z,y)=i$},\\
0&\mbox{ otherwise.~~~~}
\end{matrix}\right. \qquad(z,y\in X,~0\le i\le D).
$$
In particular, $A_0=I$ and $A_1=A$. These matrices play a key role in the study of distance-regularity. A matrix $A\in\Mat_X(\CC)$ is said to be a {\em reducible} when there exists a permutation matrix $P$ such that $P^\top A P=\left(
\begin{matrix}
X&Y\\ \O&Z
\end{matrix}
\right)$, where $X$ and $Z$ are both square, and $\O$ is a zero matrix. Otherwise, $A$ is said to be {\em irreducible}.

\begin{lemma}[{see, for example, \cite[Section~8.3]{MCm}}]
\label{gb}
A digraph $\G$ with adjacency matrix $A$ is strongly connected if and only if $A$ is an irreducible matrix.
\end{lemma}


\begin{theorem}[Perron--Frobenius Theorem]
\label{gc}
Let $\G=\G(A)$ denote a simple strongly connected digraph with spectrum $\spec(\G)$. If $\theta=\max_{\lambda\in\spec(\G)}|\lambda|$, then the following hold.
\begin{enumerate}[label=\rm(\roman*)]
\item $\theta\in\spec(\G)$.
\item The algebraic multiplicity of $\theta$ is equal to $1$.
\item There exists an eigenvector $\nnu$ with all positive entries, such that $A\nnu = \theta \nnu$.
\end{enumerate}
Sometimes it is useful to normalize a vector $\nnu$ from {\rm (iii)} in such a way that the smallest entry is equal to $1$. Such a vector $\nnu$ is called a Perron--Frobenius eigenvector.
\end{theorem}

\begin{proof}
Routine using Lemma~\ref{gb}. (See, for example, \cite[Section~8.3]{MCm}.)
\end{proof}

\smallskip
A matrix $A\in\Mat_X(\CC)$ is called {\em normal} if it commutes with its adjoint, i.e. if $A\ol{A}^\top=\ol{A}^\top A$.

\begin{theorem}[{see, for example, \cite[Chapter~7]{AS}}]
\label{gd}
Let $A\in\Mat_X(\CC)$ denote a matrix over $\CC$, with rows and columns indexed by $X$. Then, the following are equivalent.
\begin{enumerate}[label=\rm(\roman*)]
\item $A$ is normal.
\item $\CC^{|X|}$ has an orthonormal basis consisting of eigenvectors of $A$.
\item $A$ is a diagonalizable matrix.
\item The algebraic multiplicity of $\lambda$ is equal to the geometric multiplicity of $\lambda$, for every eigenvalue $\lambda$ of $A$.
\end{enumerate}
\end{theorem}

Two matrices $A,B\in\Mat_X(\CC)$ are said to be {\em simultaneously diagonalizable} if there is a nonsingular $S\in\Mat_X(\CC)$ such that $S^{-1}AS$ and $S^{-1}BS$ are both diagonal.

\begin{lemma}[{\cite[Theorem 1.3.12]{HJ}}]
\label{ge}
Two diagonalizable matrices are simultaneously diagonalizable if and only if they commute.
\end{lemma}

\begin{theorem}
\label{gf}
Let $\M$ denote a space of commutative normal matrices. Then, there exists a unitary matrix $U\in\Mat_X(\CC)$ which diagonalizes $\M$.
\end{theorem}

\begin{proof}
Immediate from Theorem~\ref{gd}, Lemma~\ref{ge} and \cite[Subsection~1.3]{HJ}.
\end{proof}

Let $\circ$ denote the elementwise--Hadamard product of matrices. We call two $01$-matrices $B$, $C$ {\em disjoint} if $B\circ C=\O$. A basis $\{F_0,F_1,\ldots,F_d\}$ of some subalgebra $\C\subset \Mat_{X}(\CC)$ satisfying all the following properties is known as the {\it standard basis} of $\C$: {\rm (i)} ${\cal F}=\{F_0,F_1(=F),\ldots,F_d\}$ is a set of mutually disjoint $(0,1)$-matrices; {\rm (ii)} the sum of some (respectively all) of these matrices gives $I$ (respectively $J$); {\rm (iii)} for each $i\in \{0,\ldots,d\}$, the transpose of $F_i$ belongs to ${\cal F}$; {\rm (iv)} the vector space spanned by ${\cal F}$ is closed under both ordinary and Hadamard multiplication;  and {\rm  (v)} each $F_i$ is a polynomial (not necessarily of degree $i$) in $F$.

\begin{theorem}[{\cite[Theorem~1]{HMc}}]
\label{Ra}
Let $\G=\G(A)$ denote a digraph with adjacency matrix $A$ and vertex set $X$. Let $\jj$ denote the all-ones vector of order $|X|$ and $J$ all-ones matrix of order $|X|$ (i.e., the matrix every column of which is $\jj$). Then, the following hold.
\begin{enumerate}[label=\rm(\roman*)]
\item There exists a polynomial $h(t)$ such that $J=h(A)$ if and only if $\G$ is strongly connected and $A\jj=A^\top\jj=k\jj$, for some positive integer $k$.
\item The unique polynomial of least degree satisfying $J=h(A)$ is $|X|s(t)/s(k)$, where $(t-k)s(t)$ is the minimal polynomial of $A$ and $k$ is the positive integer such that $A\jj=k\jj$.
\end{enumerate}
\end{theorem}

\begin{lemma}
\label{on}
Let $A$ denote a complex matrix. Then, $A$ is normal if and only if $\ol{A}^\top = p(A)$ for some polynomial $p\in\CC[t]$.
\end{lemma}

\begin{proof}
$(\La)$ Assume that $\ol{A}^\top = p(A)$ for some polynomial $p\in\CC[t]$. Then, $A\ol{A}^\top=Ap(A)=p(A)A=\ol{A}^\top A$, and thus $A$ is a normal matrix.

$(\Ra)$ Assume that $A$ is a normal matrix, and let $\spec(A)=\{[\lambda_0]^{m(\lambda_0)},\ldots,[\lambda_d]^{m(\lambda_d)}\}$ denote the spectrum of $A$. For each eigenvalue $\lambda_i$ $(0\le i\le d)$ let $U_i$ denote the matrix whose columns form an orthonormal basis for the eigenspace $\ker(A-\lambda_i I)$, and note that $\dim(\ker(A-\lambda_i I))=m(\lambda_i)$. Abbreviate $m_i=m(\lambda_i)$ $(0\le i\le d)$. Pick $i$ $(0\le i\le d)$, and note that
$
AU_i=
\lambda_iU_i.
$
If $U=[U_1|U_2|...|U_d]$, then
$$
A=U\Lambda\ol{U}^{\top},
\qquad
\mbox{ where }
\qquad
\Lambda=\left[
\begin{matrix} 
\lambda_{0} I_{m_0} & 0 & ... & 0 \\
0 & \lambda_{1} I_{m_1} & ... & 0 \\
\vdots & \vdots & ~ & \vdots \\
0 & 0& ... & \lambda_{d} I_{m_d}
\end{matrix}
\right].
$$
Now, it is routine to see that
\begin{align*}
A&=U\Lambda\ol{U}^{\top}=
[U_0|U_1|\ldots|U_d]
\left[
\begin{matrix} 
\lambda_0I_{m_0} & 0 & \cdots & 0 \\
0 & \lambda_1I_{m_1} & \cdots & 0 \\
\vdots & \vdots & \ddots & \vdots \\
0 & 0 & \ldots & \lambda_dI_{m_d} \\
\end{matrix}
\right]
\left[
\begin{matrix} 
\underline{\ol{U_0}^{\top}}\\
\underline{\ol{U_1}^{\top}}\\
\underline{\,\,\,\vdots\,\,\,}\\
\ol{U_d}^{\top}\\
\end{matrix}
\right]\\
&=[\lambda_0U_0|\lambda_1U_1|\cdots|\lambda_dU_d]
\left[
\begin{matrix} 
\underline{\ol{U_0}^{\top}}\\
\underline{\ol{U_1}^{\top}}\\
\underline{\,\,\,\vdots\,\,\,}\\
\ol{U_d}^{\top}\\
\end{matrix}
\right]\\
&=\lambda_0U_0\ol{U_0}^{\top}+\lambda_1 U_1\ol{U_1}^{\top}+\cdots+ \lambda_dU_d\ol{U_d}^{\top}\\
&=\lambda_0E_0+\lambda_1E_1+\cdots+\lambda_d E_d\qquad(\mbox{where }E_i:=U_i\ol{U_i}^{\top}).
\end{align*}
Since $\ol{E_i}^\top=E_i$, it follows that $\overline{A}^\top=\ol{\lambda_0} E_0+\ol{\lambda_1} E_1+\cdots+\ol{\lambda_d} E_d$. Note that $\{E_0,E_1,\ldots,E_d\}$ is a basis of $\A=\Span\{A^0,A^1,\ldots,A^d\}$. The result follows.
\end{proof}

\subsection{Commutative association schemes}

In Section~\ref{1A}, we have already provided the definition of a commutative $d$-class association scheme $\XXi=\{X,\{R_i\}_{i=0}^d\}$ together with definitions of relations $\{R_i\}_{i=0}^d$, relation matrices $\{B_i\}_{i=0}^d$ and Bose--Mesner algebra $\M$ of $\XXi$. The meaning of a ``matrix generates $\XXi$'' has been given in Section~\ref{1A} as well. Note that relation matrices $\{B_i\}_{i=0}^d$ form a standard basis of the Bose-Mesner algebra $\M$. We say that the relation $R_i$ {\em generates} the association scheme $\XXi$ if every element of the Bose--Mesner algebra $\M$ of $\XXi$ can be writen as a polynomial in $B_i$, where $B_i$ is adjacency matrix of the (di)graph $(X,R_i)$. We say that the association scheme $\XXi$ is {\em $P$-polynomial} with respect to $B_1$, if it is generated by $B_1$ and there exists an ordering $(B_0,B_1,\ldots,B_d)$ and polynomials $p_j(t)$ of degree $j$, such that $B_j=p_j(B_1)$ $(0\le j\le d)$. In Section~\ref{dA}, we show that if $(B_0,B_1,\ldots,B_d)$ is a $P$-polynomial ordering of $\XXi$ then $\G=\G(B_1)$ is a distance-regular (di)graph (see Proposition~\ref{Oc}). Recall that the association scheme $\XXi$ is {\em polynomial} (with respect to $R_i$) if it is generated by a relation $R_i$ for some $i$ (we recommend articles \cite{MPc,TtYl} for interesting results in that direction).


\begin{proposition}
\label{Ha}
Let $\G=\G(A)$ denote a digraph with adjacency matrix $A$ and vertex set $X$. Assume that $A$ generates the Bose--Mesner algebra $\M$ of a commutative $d$-class association scheme, and let $\{B_0,B_1,\ldots,B_d\}$ denote the standard basis of $\M$. Then, the following hold.
\begin{enumerate}[label=\rm(\roman*)]
\item For any $i$ $(0\le i\le d)$ and $y,z,u,v\in X$, if $(B_i)_{zy}=(B_i)_{uv}=1$ then $\partial(z,y)=\partial(u,v)$.
\item Every distance-$i$ matrix $A_i$ of $\G=\G(A)$ belongs to $\M$, i.e., $A_i\in\M$ $(0\le i\le D)$.
\end{enumerate}
\end{proposition}

\begin{proof}
Since $A$ generates the Bose--Mesner algebra $\M$, and $J\in\M$, there exists a polynomial $p(t)$ such that $J=p(A)$. This implies that $\G$ is regular and strongly connected (see Theorem~\ref{Ra}).

(i) For every $\ell\in\NN$, there exists complex scalars $\alpha^{(\ell)}_i$ $(0\le i\le d)$ such that $A^\ell=\sum_{i=0}^{d} \alpha^{(\ell)}_i B_i$. Recall that $\sum_{i=0}^d B_i = J$ and $B_i\circ B_j=\delta_{ij}B_i$ $(0\le i,j\le d)$. This yields that, for any $y,z,u,v\in X$ and $i$ $(0\le i\le d)$, if $(B_i)_{zy}\ne 0$ and $(B_i)_{uv}\ne 0$ then $(A^\ell)_{zy}=(A^\ell)_{uv}=\alpha^{(\ell)}_i$, i.e., the number of walks of length $\ell$ from $z$ to $y$ is equal to the the number of walks of length $\ell$ from $u$ to $v$ (see Lemma~\ref{hB}). Moreover, $(A^\ell)_{zy}=(A^\ell)_{uv}$ holds for any $\ell$ $(\ell\in\NN)$. To prove the claim, we use the proof by a contradiction, similar as in \cite[Lemma~2.3]{FQpG} where the author has an undirected graph. Assume that $\partial(z,y)>\partial(u,v)=m$. Then, $(A^m)_{uv}\ne 0$ and $(A^m)_{zy}=0$, a contradiction. 
The result follows. 

(ii) From the proof of (i) above it follows that, if $y,z\in X$ are two arbitrary vertices such that $\partial(z,y)=i$, then there exists $B_j$ (for some $0\le j\le d$) such that $(B_j)_{zy}=1$. Recall also that $(A_i)_{zy}=1$. In fact, for such a choice of $j$ and any nonzero $(u,v)$-entry of $B_j$, we have $\partial(u,v)=i$. This yields
$$
A_i=\sum_{j:A_i\circ B_j\ne{\boldsymbol{O}}} B_j
\qquad(0\le i\le D).
$$
The result follows.
\end{proof}

Through the rest of the paper, we use the following notation.

\begin{notation}{\label{dC}}{\rm
Let $\G$ denote a simple strongly connected digraph with vertex set $X$,  diameter $D$, and let $\{A_0,A:=A_1,A_2,\ldots,A_D\}$ denote the distance-$i$ matrices of $\G=\G(A)$. Assume that the adjacency matrix $A$ has $d+1$ distinct eigenvalues, and let $\A$ denote the adjacency algebra of $\G$ (i.e., the algebra generated by the adjacency matrix $A$ with respect to the ordinary matrix multiplication). Pick $x\in X$, and define the sets $\G_i^\ra(x)$ and $\G_j^\la(x)$  $(0\le i,j\le D)$ in the following way
$$
\G^\ra_i(x)=\{z\in X\mid\partial(x,z)=i\}
\qquad\mbox{and}\qquad
\G^\la_j(x)=\{z\in X\mid\partial(z,x)=j\}.
$$
}\end{notation}

\subsection{Results of {\sc Comellas}, {\sc Fiol}, {\sc Gimbert} and {\sc Mitjana}}
\label{zm}

In \cite{CfFgM}, {\sc Comellas}, {\sc Fiol}, {\sc Gimbert} and {\sc Mitjana} had studied digraphs with the following purely combinatorial property: for each non-negative integer $\ell$ $(0\le\ell\le D)$, the number of walks of length $\ell$ in $\G$ between two vertices $x,y\in X$ only depends on $h=\partial(x,y)$ (distance from $x$ to $y$) (we call these graphs weakly distance-regular in sense of {\sc Comellas et al.} \cite{CfFgM}). Among else, the authors gave several algebraic characterizations of these graphs, and they made connection with digraphs that are distance-regular in sense of {\sc Damerell} {\rm \cite{Drm}}.

\begin{proposition}[{\cite[Theorem~2.2]{CfFgM}}]
\label{zi}
Let $\G$ denote a connected digraph with diameter $D$ and distance-$i$ matrices $\{A_i\}_{i=0}^D$. Then the following {\rm (i)--(iv)} are equivalent.
\begin{enumerate}[label=\rm(\roman*)]
\item For each non-negative integer $\ell$ $(0\le\ell\le D)$, the number of walks of length $\ell$ in $\G$ between two vertices $x,y\in X$ only depends on $h=\partial(x,y)$ (distance from $x$ to $y$) ($\G$ is weakly distance-regular in sense of {\sc Comellas et al.}).
\item The distance-$i$ matrix $A_i$ is a polynomial of degree $i$ in the adjacency matrix $A$.
\item The set of distance-$i$ matrices $\{A_i\}_{i=0}^D$ is a basis of the adjacency algebra $\A$ of $\G$.
\item For any two vertices $u,v\in X$ at distance $\partial(u,v)=h$, the numbers $p^h_{ij}(u,v):=|\G_i^\ra(u)\cap\G_j^\la(v)|$ do not depend on the vertices $u$ and $v$, but only on their distance $h$.
\end{enumerate}
\end{proposition}

\begin{proposition}[{\cite[Proposition~2.6]{CfFgM}}]
\label{zj}
Let $\G$ denote a connected digraph with diameter $D$ and distance-$i$ matrices $\{A_i\}_{i=0}^D$. Assume that {\rm (i)-(iv)} from Proposition~\ref{zi} hold. Then $A$ is a normal matrix if and only if there exists numbers $b_{ij}$ $(0\le i,j\le D)$ such that $|\G^\ra_1(y)\cap\G_j^\ra(x)|=b_{ij}$, for all $x\in X$, $y\in\G^\ra_i(x)$ $(0\le i,j\le D)$ (i.e., $\G$ is a distance-regular digraph in sense of {\sc Damerell} {\rm \cite{Drm}}).
\end{proposition}

\subsection{One more result that we use later in the paper}
\label{eb}

It is known that the number of connected components of a regular graph is the multiplicity of its valency (see \cite[page~1]{BVM}). For the sake of Section~\ref{qa} of the paper we need a similar result for digraphs. We did not manage to find similar result implicitly or explicitly in literature. Recall that digraph is $k$-regular, if $\G^{\la}_1(x)=\G^{\ra}_1(x)=k$ holds for every $x\in X$.

\begin{lemma}
\label{ea}
Let $\G$ denote a simple strongly connected $k$-regular digraph. Then $k$ is an eigenvalue of $\G$ of algebraic multiplicity $1$.
\end{lemma}

\begin{proof}
A square matrix $M$ is {\em stochastic} if all of its entries are nonnegative, and the entries of each column sum to $1$. By \cite[Subsection~5.6]{mR}, if $M$ is a stochastic matrix, then $1$ is an eigenvalue of $M$; and if $\lambda$ is a (real or complex) eigenvalue of $A$, then $|\lambda|\le 1$. 

In our case, we have that $k^{-1}A$ is a stochastic matrix. It is routine to show that $\eig(\G)$ 
are eigenvalues of $A$ if and only if $k^{-1}\eig(\G)$ are eigenvalues of $k^{-1}A$. It follows that $k=\max_{\lambda\in\eig(\G)}|\lambda|$. By Theorem~\ref{gc}, the algebraic multiplicity of $k$ is equal to $1$.
\end{proof}

\begin{proposition}
\label{ec}
Let $\G$ denote a simple directed $k$-regular graph with adjacency matrix $A$. If $A$ is a diagonalizable matrix, then $\G$ has $\ell$ connected components if and only if the multiplicity of the eigenvalue $k$ is $\ell$.
\end{proposition}

\begin{proof}
Since $\G$ is a $k$-regular digraph $A\jj=A^\top\jj=k\jj$. For the moment assume that $\G$ has $r$ connected components $C_i$ $(0\le i\le r)$, and permute the index set of the adjacency matrix $A$ in such a way that 
$$
A=\begin{blockarray}{cccccc}
~ & C_1 & C_2 & C_3 & \cdots & C_r \\
\begin{block}{c(ccccc)}
C_1&A_1&\O&\O&\cdots&\O\\
C_2&\O&A_2&\O&\cdots&\O\\
\vdots&\vdots&\vdots&\vdots& ~ &\vdots\\
C_r&\O&\O&\O&\cdots&A_r\\
\end{block}
\end{blockarray}
$$
where $A_i$ is the adjacency matrix of the component $C_i$ $(0\le i\le r)$. Note that each component $C_i$ $(0\le i\le r)$ is a strongly connected $k$-regular digraph. Since $A$ is a diagonalizable matrix, there exists an orthonormal basis $\B$ consisting of the eigenvectors of $A$; let $P$ denote a matrix with columns vectors the elements from $\B$. Then, $A=P\Lambda\ol{P}^\top$ where $\Lambda$ is a diagonal matrix whose entries on the main diagonal are the eigenvalues of $\G$. Moreover, we have
$$
\ol{P}^\top AP=
\begin{blockarray}{cccccc}
~ & C_1 & C_2 & C_3 & \cdots & C_r \\
\begin{block}{c(ccccc)}
C_1&\Lambda_1&\O&\O&\cdots&\O\\
C_2&\O&\Lambda_2&\O&\cdots&\O\\
\vdots&\vdots&\vdots&\vdots& ~ &\vdots\\
C_r&\O&\O&\O&\cdots&\Lambda_r\\
\end{block}
\end{blockarray}
$$
where $\Lambda_i$ $(1\le i\le r)$ is a diagonal matrix whose entries on the main diagonal are the eigenvalues of $A_i$. By Lemma~\ref{ea}, every $\Lambda_i$ has exactly one entry that is equal to $k$ at main diagonal. The result follows.
\end{proof}



\section{Some algebraic properties of distance-$\boldsymbol{i}$ matrices of a strongly connected digraph}
\label{Av}

In this section we prove some claims that we use in the rest of the paper. Some of claims from this section can be found implicitly (or explicitly) in the literature. Without this section our paper is not readable as we want it to be. Also, this section can be used as a ``warm-up'' in working with the structure of a particular family of digraphs, i.e., the family that satisfies conditions (i), (ii) of Problem~\ref{Aa}. Recall that the algebra generated by the adjacency matrix $A$ of $\G$ with respect to the ordinary matrix multiplication is known as {\em adjacency algebra} of $\G$, and this algebra is denoted by $\A$.

\begin{proposition}
\label{eJ}
With reference to {\rm Notation~\ref{dC}}, assume that $A^\top\in\{A_0,A_1,\ldots,A_D\}$. If there exist polynomials $p_i(t)\in\RR[t]$ $(0\le i\le D)$ such that $A_i=p_i(A)$, then the following hold.
\begin{enumerate}[label=\rm(\roman*)]
\item $A$ is a normal matrix.
\item $\dim(\A)=d+1$ and $\A=\Span\{A^0,A^1,\ldots,A^d\}$.
\end{enumerate}
\end{proposition}

\begin{proof}
(i) Since $A^\top\in\{A_0,A_1,\ldots,A_D\}$ holds by assumption, there exists $h$ $(0\le h\le D)$ such that $A^\top=A_h$. Now we have 
$$
A\ol{A}^\top=AA_h=Ap_h(A)=p_h(A)A=A_hA=\ol{A}^\top A,
$$ 
which yields that $A$ is normal matrix. 

(ii) Since $A$ is a normal matrix, $\CC^{|X|}$ has an orthonormal basis consisting of eigenvectors of $A$ (see, for example, \cite[Chapter~7]{LADr}). For each eigenvalue $\lambda_i$ $(0\le i\le d)$ of $A$, let $U_i$ be the matrix whose columns form an orthonormal basis of the eigenspace $\ker(A-\lambda_iI)$, and define some matrices $E_i$ $(0\le i\le d)$ in the following way 
$$
E_i=U_i\ol{U}_i^\top.
$$ 
They are nonzero since $\trace(E_i)$ is equal to multiplicity of $\lambda_i$ (i.e., $\trace(E_i)=\trace(U_i\ol{U}_i^\top)=\trace{(\ol{U}_i^\top U_i)}=m(\lambda_i)$). From this definition of the $E_i$'s, it follows $E_iE_j=\delta_{ij}E_i$ $(0\le i,j\le d)$, and using that property it is routine to prove that $\{E_i\}_{i=0}^d$ is a linearly independent set (equation $\sum_{i=0}^d \alpha_i E_i=\O$ yields $\alpha_i=0$ for $(0\le i\le d)$). Now, in the same way as in \cite[Lemma~2.2]{SP}, we can prove that $A^h=\sum_{i=0}^d \lambda_i^h E_i$ $(0\le h\le d)$. This system can be written as
$$
\left [
\begin{matrix} 
I\\ A\\ A^2\\ \vdots\\ A^d
\end{matrix}
\right]=
{\underbrace{
\left[
\begin{matrix} 
1 & 1 & \cdots & 1\\
\lambda_0 & \lambda_1 & \cdots & \lambda_d\\
\lambda_0^2 & \lambda_1^2 & \cdots & \lambda_d^2\\
\vdots & \vdots & \, & \vdots \\
\lambda_0^d & \lambda_1^d & \cdots & \lambda_d^d\\
\end{matrix}
\right]}_{=B^\top}}
\left[
\begin{matrix} 
E_0\\ E_1\\ E_2\\ \vdots\\ E_d
\end{matrix}
\right].
$$
where $B$ is a Vandermonde matrix (see, for example, \cite[page 185]{MCm}), which is invertible. The result follows.
\end{proof}

\begin{lemma}
\label{eK}
Let $\G$ denote a simple strongly connected digraph of diameter $2$ with vertex set $X$, and let $\{A_0,A:=A_1,A_2\}$ denote the distance-$i$ matrices of $\G$ with the properties that $A^\top=A_2$ and $A\jj=A_2\jj=k\jj$. Then, there exists a polynomial $p(t)\in\RR[t]$ such that $A_2=p(A)$.
\end{lemma}

\begin{proof}
Since $A\jj=A^\top\jj=k\jj$, by Theorem~\ref{Ra} it follows that there exists a polynomial $h(t)\in\RR[t]$ such that $J=h(A)$. Now we have $A_2=h(A)-A-I$, and the result follows.
\end{proof}

\begin{proposition}
\label{Cg}
Let $\G$ denote a simple strongly connected digraph with vertex set $X$,  diameter $D$, and let $\{A_0,A:=A_1,A_2,\ldots,A_D\}$ denote the distance-$i$ matrices of $\G$. If there exist complex scalars $p^h_{i1}$ $(0\le i,j,h\le D)$ such that 
\begin{equation}
\label{eG}
A_iA=\sum_{h=0}^D p^h_{i1} A_h
\qquad(0\le i,j\le D),
\end{equation}
then for the scalars $p^j_{i1}$ $(0\le i,j\le D)$ the following hold.
\begin{enumerate}[label=\rm(\roman*)]
\item $p^j_{i1}=0$, if $j>i+1$.
\item $p^j_{i1}\ne 0$, if $j=i+1$.
\end{enumerate}
Moreover, 
\begin{enumerate}[start=3,label=\rm(\roman*)]
\item There exist polynomials $p_i$ with $\deg p_i\le i$ $(0\le i\le D)$ such that $A_i=p_i(A)$.
\item $A_iA_j=A_jA_i$ $(0\le i,j\le d)$.
\end{enumerate}
\end{proposition}

\begin{proof}
Pick $x\in X$ and consider distance-$i$ partition $\{\G_i^\ra(x)\}_{i=0}^D$ of the graph $\G$. Fix $i$ $(0\le i\le D)$, and note that for any $y\in\G_{j}^\ra(x)$ $(0\le j\le D)$ we have
\begin{align*}
p^j_{i1} &= p^j_{i1} (A_j)_{xy}\\
&=\left( \sum_{h=0}^D p^h_{i1} A_h \right)_{xy}\\
&= (A_iA)_{xy}\\
&=\sum_{z\in X} (A_i)_{xz} (A)_{zy}\\
&=|\G_i^\ra(x)\cap\G_1^\la(y)|.
\end{align*}
We get 
$$
|\G_i^\ra(x)\cap\G_1^\la(y)|=p^{j}_{i1}\qquad (\forall y\in\G^\ra_j(x)).
$$

(i) Our proof is by a contradiction. Assume that, there exist $j>i+1$ such that $p^j_{i1}\ne 0$. This yields that there exists $y\in\G_j^\ra(x)$ such that $\underbrace{|\G_i^\ra(x)\cap\G_1^\la(y)|}_{=p^j_{i1}}\ne 0$, which implies that there exists $z\in\G_i^{\ra}(x)$ such that $\partial(z,y)=1$. This now yields that $y\in\bigcup_{h=0}^{i+1} \G_h^\ra(x)$, i.e., $y\notin\G_{j}^\ra(x)$ (since $j>i+1$), a contradiction.

(ii) Note that for any $y\in\G_{i+1}^\ra(x)$
$$
|\G_i^\ra(x)\cap\G_1^\la(y)|=p^{i+1}_{i1}.
$$
On the other hand, by the structure of distance-$i$ partition $\{\G_i^\ra(x)\}_{i=0}^D$ (see Figure~\ref{01}), the cardinality of the set $\G_i^\ra(x)\cap\G_1^\la(y)$, for any  $y\in\G_{i+1}^\ra(x)$, is at least $1$. Thus, $p^{i+1}_{i1}\ne 0$.

(iii) From claim (i) we have that \eqref{eG} can be written in the following way
\begin{equation}
\label{eH}
A_iA=\sum_{h=0}^{i+1} p^h_{i1} A_h.
\end{equation}
Now the proof is the same as the proof of \cite[Proposition~2.7.1]{BCN}, where the authors considered the case of symmetric association schemes. Since $p^{i+1}_{1i}\ne 0$, $A_{i+1}$ is a linear combination of $A_0,A_1,\ldots,A_i$ and $A_iA$. Using induction on $i$, $A_{i+1}$ can be written as a polynomial of degree $i+1$ in $A$. The result follows.

(iv) From the claim (iii) we have that there exists a sequence $\{p_i(t)\}_{i=0}^D$ of polynomials $p_i(t)\in\RR[t]$ $(0\le i\le d)$ of degree $i$ such that $p_i(A)=A_i$. Note that 
$$
A_iA_j=p_i(A)p_j(A)=|\mbox{some polynomial in $A$}|=p_j(A)p_i(A)=A_jA_i.
$$
The claim (iv) follows.
\end{proof}

\begin{corollary}
\label{eF}
Let $\G$ denote a simple strongly connected digraph with vertex set $X$,  diameter $D$, and let $\{A_0,A:=A_1,A_2,\ldots,A_D\}$ denote the distance-$i$ matrices of $\G$. If there exist complex scalars $p^h_{i1}$ $(0\le i,h\le D)$ such that 
\begin{equation}
\label{ef}
A_iA=\sum_{h=0}^D p^h_{i1} A_h
\qquad(0\le i,j\le D),
\end{equation}
then the following hold.
\begin{enumerate}[label=\rm(\roman*)]
\item The vector space $\Span\{A_0,A_1,\ldots,A_D\}$ is closed with respect to the ordinary matrix multiplication, i.e., there exist complex scalars (known as {\em intersection numbers}) $p^h_{ij}$ $(0\le i,j,h\le D)$ such that 
$$
A_iA_j=\sum_{h=0}^D p^h_{ij} A_h
\qquad(0\le i,j\le D).
$$
\end{enumerate}
Moreover, for the intersection numbers $p^h_{ij}$ $(0\le i,j,h\le D)$ the following hold.
\begin{enumerate}[start=2, label=\rm(\roman*)]
\item $p^h_{ij}=0$, if $h>i+j$.
\item $p^h_{ij}\ne 0$, if $h=i+j$.
\end{enumerate}
\end{corollary}

\begin{proof} From Proposition~\ref{Cg}(iii) we have that there exists a sequence $\{p_i(t)\}_{i=0}^D$ of polynomials $p_i(t)\in\RR[t]$ $(0\le i\le d)$ of degree $i$ such that $p_i(A)=A_i$.

(i) For the moment, let $\M=\Span\{A_0,A_1,\ldots,A_D\}$. By \eqref{ef} we have that $A^\ell\in\M$ for every $\ell\in\NN$, which yields $\A\subseteq\M$. By Proposition~\ref{Cg}(iii), we also have $\M\subseteq\A$. The result follows.

(ii), (iii). The proof is the same as the proof of \cite[Proposition~2.7.1]{BCN}, where the authors considered the case of symmetric association schemes. We have 
$$
p_i(t)p_j(t) = \sum_{h=0}^D p^h_{ij} p_h(t),
$$
so that $p^h_{ij}=0$ for $h>i+j$ and $p^h_{ij}\ne 0$ for $h=i+j$ (comparing polynomials degrees).  
\end{proof}

\medskip
From Corollary~\ref{eF} imminently the following question arises: Under which restriction we have $A_i^\top\in\{A_0,A_1,\ldots,A_D\}$ $(0\le i\le D)$? Can we prove that $A^\top\in \{A_0,A_1,\ldots,A_D\}$ implies $A_i^\top\in\{A_0,A_1,\ldots,A_D\}$ $(0\le i\le D)$? We give an answer in Proposition~\ref{eL}.

\begin{proposition}
\label{eL}
Let $\G$ denote a simple strongly connected digraph with vertex set $X$,  diameter $D$, and let $\{A_0,A:=A_1,A_2,\ldots,A_D\}$ denote the distance-$i$ matrices of $\G$. Assume that properties {\rm (i), (ii)} of {\rm Problem~\ref{Aa}} hold, i.e., assume that there exist complex scalars $p^h_{i1}$ $(0\le i,h\le D)$ such that 
\begin{equation}
\label{eO}
A_iA=\sum_{h=0}^D p^h_{i1} A_h
\qquad(0\le i\le D),
\end{equation}
and that $A^\top\in\{A_0,A_1,\ldots,A_D\}$. Then, the following hold.
\begin{enumerate}[label=\rm(\roman*)]
\item $A$ has $D+1$ distinct eigenvalues.
\item $A_i^\top\in\{A_0,A_1,\ldots,A_D\}$ $(0\le i\le D)$.
\end{enumerate}
\end{proposition}

\begin{proof} From Proposition~\ref{eJ}(i), $A$ is a normal matrix.  Assume that $A$ has $d+1$ distinct eigenvalues.

(i) We show that $D=d$. Consider the vector space $\M=\Span\{A_0,A_1,\ldots,A_D\}$. By Corollary~\ref{eF}(i), $\M$ is closed with respect to the ordinary matrix multiplication. By Proposition~\ref{eJ}, $\dim(\A)=d+1$. Since each of the $A_i$'s can be written as polynomial in $A$, we have $\M\subseteq\A$ and consequently $D\le d$. On the other hand \eqref{eO} yields that $\A\subseteq\M$ and consequently $d\le D$. The result follows.

(ii) In the proof of claim (i), we showed two things (a) $D=d$; and (b) $\A=\Span\{A^0,A^1,\ldots,\break A^D\}=\Span\{A_0,A_1,\ldots,A_D\}$. Since $A^\top\in\{A_i\}_{i=0}^D$, let $h$ denote the index such that $A^\top=A_h$. For any $i$ $(0\le i\le D)$ we have
\begin{align*}
A_i^\top &= (p_i(A))^\top\\
&=p_i(A^\top)\\
&=p_i(A_h)\in\A.
\end{align*}
In the end, since $\{A_0^\top,A_1^\top,\ldots,A_D^\top\}$ is a linearly independent subset of $D+1$ elements of $\A$$(=\M)$, which are also $\circ$-idempotents, $\{A_i\}_{i=0}^D=\{A_i^\top\}_{i=0}^D$ and the result follows.
\end{proof}


\section{Definition of a distance-regular digraph}
\label{Ga}

Let $\G$ denote a graph with diameter $D$. In this section, we give a definition of a distance-regular digraph and we prove some claims that we use in the end of the paper. Most interesting result is in Proposition~\ref{dF}: if $\G$ is a distance-regular digraph, then for any vertex $x$ the partition $\{\G_0^{\la}(x),\G_1^{\la}(x),\ldots,\G_D^{\la}(x)\}$ is equitable, and 
$$
\{\G_0^{\ra}(x),\G_1^{\ra}(x),\ldots,\G_D^{\ra}(x)\}=
\{\G_0^{\la}(x),\G_1^{\la}(x),\ldots,\G_D^{\la}(x)\}.
$$

Motivation for Definition~\ref{jA} is the following Theorem~\ref{pi}.

\begin{theorem}[{\cite[Theorem~1.2]{MPc}}]
\label{pi}
Let $\M$ denote the Bose--Mesner algebra of a commutative $d$-class association scheme $\XXi=(X,\R)$, and $A\in\M$ denote a $01$-matrix. Assume that $\G=\G(A)$ denotes a (strongly) connected (directed) graph. Then the following hold.
\begin{enumerate}[label=\rm(\roman*)]
\item For every vertex $x\in X$, there exists an $x$-distance-faithful intersection diagram (of an equitable partition $\Pi_x$) with $d+1$ cells.
\item The structure of the $x$-distance-faithful intersection diagram (of the equitable partition $\Pi_x$) from {\rm (i)} does not depend on $x$.
\end{enumerate}
\end{theorem}

The following result of {\sc Damerell} is well known.
\begin{lemma}[{\cite[Theorem~3]{Drm}}]
\label{PI}
Let $\G$ denote a strongly connected digraph with vertex set $X$. Assume that there exists numbers $b_{ij}$ $(0\le i,j\le D)$ such that $|\G^\ra_1(y)\cap\G_j^\ra(x)|=b_{ij}$, for all $x\in X$, $y\in\G^\ra_i(x)$ $(0\le i,j\le D)$ (i.e., $\G$ is a distance-regular digraph in sense of {\sc Damerell} {\rm \cite{Drm}}). Then $(X,\{R_i\}_{i=0}^D)$ is a commutative association scheme, where $\{R_i\}_{i=0}^D$ denotes a partition of $X\times X$ with $R_i=\{(x,y)\in X\times X\mid (A_i)_{xy}=1\}$ $(0\le i\le D)$. 
\end{lemma}

\begin{definition}
\label{jA}{\rm
With reference to {\rm Notation~\ref{dC}}, pick $x\in X$. We say that strongly connected digraph $\G$ is {\em distance-regular around $x$} (or {\em distance-regularized around $x$}) if there are integers $d^\ra_{ij}(x)$, $d^\la_{ij}(x)$ $(0\le i,j\le D)$ which satisfy
$$
|\G_1^{\ra}(y)\cap\G^\ra_j(x)|=d^\ra_{ij}(x) \qquad(0\le i,j\le D)
$$
and
$$
|\G_1^{\la}(y)\cap\G^\ra_j(x)|=d^\la_{ij}(x) \qquad(0\le i,j\le D)
$$
for any $y\in\G_i^\ra(x)$.
}\end{definition}

\medskip
From Definition~\ref{jA}, note that $\G$ is directed distance-regular around $x$, if for each $y\in\G^\ra_i(x)$ $(0\le i\le D)$, the numbers
\begin{equation}
\label{jB}
d^\ra_{ij}(x) := |\G_1^{\ra}(y)\cap\G^\ra_j(x)|
\quad\mbox{ and }\quad
d^\la_{ij}(x) := |\G_1^{\la}(y)\cap\G^\ra_j(x)|
\qquad(0\le i,j\le D)
\end{equation}
depend on the index $j$ and the distance $i$ from $x$ to $y$, but not on the choice of $y\in\G_i^\ra(x)$. Moreover, the definition of distance-regular digraph around $x$ is equivalent to the following claim: for the vertex $x$, the distance-$i$ partition 
$$
\{\G^\ra_0(x),\G^\ra_1(x),\ldots,\G^\ra_D(x)\}
$$
of $X$ is equitable with corresponding parameters $d^\ra_{ij}(x)$ and $d^\la_{ij}(x)$ (see Figure~\ref{01}).

{\small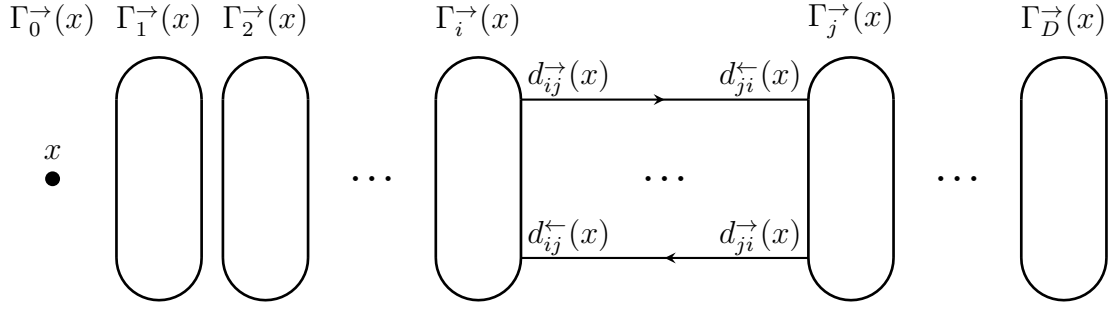
\begin{figure}[t]{\rm
\begin{center}
\begin{tikzpicture}[scale=0.7]
\draw [line width=1pt] (1.2,5)-- (1.2,2);
\draw [line width=1pt] (2.8,5)-- (2.8,2);
\draw [shift={(2,5)},line width=1pt]  plot[domain=0:3.141592653589793,variable=\t]({1*0.8*cos(\t r)+0*0.8*sin(\t r)},{0*0.8*cos(\t r)+1*0.8*sin(\t r)});
\draw [shift={(2,2)},line width=1pt]  plot[domain=3.141592653589793:6.283185307179586,variable=\t]({1*0.8*cos(\t r)+0*0.8*sin(\t r)},{0*0.8*cos(\t r)+1*0.8*sin(\t r)});
\draw [shift={(4,5)},line width=1pt]  plot[domain=0:3.141592653589793,variable=\t]({1*0.8*cos(\t r)+0*0.8*sin(\t r)},{0*0.8*cos(\t r)+1*0.8*sin(\t r)});
\draw [shift={(4,2)},line width=1pt]  plot[domain=3.141592653589793:6.283185307179586,variable=\t]({1*0.8*cos(\t r)+0*0.8*sin(\t r)},{0*0.8*cos(\t r)+1*0.8*sin(\t r)});
\draw [line width=1pt] (3.2,5)-- (3.2,2);
\draw [line width=1pt] (4.8,5)-- (4.8,2);
\draw [line width=1pt] (7.2,5)-- (7.2,2);
\draw [line width=1pt] (8.8,5)-- (8.8,2);
\draw [line width=1pt] (14.2,5)-- (14.2,2);
\draw [line width=1pt] (15.8,5)-- (15.8,2);
\draw [line width=1pt] (18.2,5)-- (18.2,2);
\draw [line width=1pt] (19.8,5)-- (19.8,2);
\draw [shift={(8,5)},line width=1pt]  plot[domain=0:3.141592653589793,variable=\t]({1*0.8*cos(\t r)+0*0.8*sin(\t r)},{0*0.8*cos(\t r)+1*0.8*sin(\t r)});
\draw [shift={(15,5)},line width=1pt]  plot[domain=0:3.141592653589793,variable=\t]({1*0.8*cos(\t r)+0*0.8*sin(\t r)},{0*0.8*cos(\t r)+1*0.8*sin(\t r)});
\draw [shift={(19,5)},line width=1pt]  plot[domain=0:3.141592653589793,variable=\t]({1*0.8*cos(\t r)+0*0.8*sin(\t r)},{0*0.8*cos(\t r)+1*0.8*sin(\t r)});
\draw [shift={(19,2)},line width=1pt]  plot[domain=3.141592653589793:6.283185307179586,variable=\t]({1*0.8*cos(\t r)+0*0.8*sin(\t r)},{0*0.8*cos(\t r)+1*0.8*sin(\t r)});
\draw [shift={(15,2)},line width=1pt]  plot[domain=3.141592653589793:6.283185307179586,variable=\t]({1*0.8*cos(\t r)+0*0.8*sin(\t r)},{0*0.8*cos(\t r)+1*0.8*sin(\t r)});
\draw [shift={(8,2)},line width=1pt]  plot[domain=3.141592653589793:6.283185307179586,variable=\t]({1*0.8*cos(\t r)+0*0.8*sin(\t r)},{0*0.8*cos(\t r)+1*0.8*sin(\t r)});
\node at (6,3.5) {$\boldsymbol\ldots$};
\node at (11.5,3.5) {$\boldsymbol\ldots$};
\node at (17,3.5) {$\boldsymbol\ldots$};
\node at (0,6.5) {$\G^{\ra}_0(x)$};
\node at (2,6.5) {$\G^{\ra}_1(x)$};
\node at (4,6.5) {$\G^{\ra}_2(x)$};
\node at (8,6.5) {$\G^{\ra}_{i}(x)$};
\node at (15,6.5) {$\G^{\ra}_{j}(x)$};
\node at (19,6.5) {$\G^{\ra}_{D}(x)$};
\node at (0,4) {$x$};
\fill (0,3.5) circle [radius=0.14];

\draw[line width=.8pt,draw=black,postaction={on each segment={mid arrow=black}}] (8.8,5) -- (14.2,5);
\draw[line width=.8pt,draw=black,postaction={on each segment={mid arrow=black}}] (14.2,2) -- (8.8,2);
\node at (9.7,5.4) {$d^\ra_{ij}(x)$};
\node at (9.7,2.4) {$d^\la_{ij}(x)$};
\node at (13.3,2.4) {$d^\ra_{ji}(x)$};
\node at (13.3,5.4) {$d^\la_{ji}(x)$};

\end{tikzpicture}
\caption{\rm 
Illustration for the distance-$i$ diagram of a distance-regular digraph around $x$. 
}
\label{01}
\end{center}
}\end{figure}}

\medskip
Using Definition~\ref{jA}, we are now ready to give a definition of a distance-regular digraph.

\begin{definition}\label{dD}{\rm
With reference to {\rm Notation~\ref{dC}}, a strongly connected digraph $\G$ is called {\em distance-regular} if the following (i), (ii) hold.
\begin{enumerate}[label=\rm(\roman*)]
\item For every vertex $x\in X$, the distance-$i$ partition 
$$
\{\G^\ra_0(x),\G^\ra_1(x),\ldots,\G^\ra_D(x)\}
$$
of $X$ is equitable (i.e., $\G$ is a distance-regular digraph around every $x$).
\item The corresponding parameters of the equitable partition from {\rm (i)} do not depend on the choice of the vertex $x$ (i.e., with reference to \eqref{jB}, there exists integers $d^\ra_{ij}$ and $d^\la_{ij}$ $(0\le i,j\le D)$ such that $d^\ra_{ij}(x)=d^\ra_{ij}$ and $d^\la_{ij}(x)=d^\la_{ij}$ for every vertex $x\in X$).
\end{enumerate}
}\end{definition}

From Definition~\ref{dD}, note that a strongly connected digraph $\G$ is  distance-regular, if for each $x\in X$ and $y\in\G^\ra_i(x)$ $(0\le i\le D)$, the numbers
$$
d^\ra_{ij} := |\G_1^{\ra}(y)\cap\G^\ra_j(x)|
\quad\mbox{ and }\quad
d^\la_{ij} := |\G_1^{\la}(y)\cap\G^\ra_j(x)|\qquad(0\le i,j\le D)
$$
depend on the index $j$ and the distance $i$ from $x$ to $y$, but not on the choice of vertices $x\in X$ and $y\in\G_i^\ra(x)$.

\begin{lemma}
\label{dE}
With reference to {\rm Notation~\ref{dC}}, let $\G$ denote a distance-regular digraph. Pick $i$ $(0\le i\le D)$, $x\in X$, $y\in\G_i^{\ra}(x)$, and let $h:=\partial(y,x)$. Then, for every $z\in\G_i^{\ra}(x)$, $\partial(z,x)=h$.
\end{lemma}

\begin{proof}
We use mathematical induction on $h$.

{\sc Basis of induction.} Case $h=0$ yields $i=0$, and the result follows. Note that $h\ne 0$ implies $i\ne 0$. Let $h=1$, i.e., pick $y\in\G_i^{\ra}(x)$ and assume that $\partial(y,x)=1$. Since the distance-$i$ partition is equitable, every vertex from $\G_i^\ra(x)$ has exactly $d_{i0}^\ra$ neighbours in $\G_0^\ra(x)=\{x\}$. To be more precise, $d_{i0}^\ra=|\G^\ra_1(z)\cap \G_0^\ra(x)|=1$ for any $z\in\G_i^\ra(x)$. The result follows.

{\sc Step of induction.} Assume now that, for any $h$ ($0\le h\le m-1$, where $m\ge 2$), the following claim holds:
$$
\mbox{if $y\in\G_i^{\ra}(x)$ and $\partial(y,x)=h$, then for every $z\in\G_i^{\ra}(x)$ $\partial(z,x)=h$ holds.}
$$
In the next few lines, we prove that our claim holds also for $m$. So, assume that $y\in\G_i^{\ra}(x)$ and $\partial(y,x)=m$. Since $\partial(y,x)=m$, there exists a shortest path $[y,w_1,\ldots,w_{m-1},w_m:=x]$ of length $m$ such that $w_j\in\G^\ra_{\ell_j}(x)$ and $\partial(w_j,x)=m-j$ $(1\le j\le m)$ with $\G^\ra_{\ell_j}(x)\ne\G_i^\ra(x)$ $(1\le j\le m-1)$. By induction assumption, for every $z\in\G^\ra_{\ell_j}(x)$ we have $\partial(z,x)=m-j$ $(1\le j\le m)$. Note that $y$ does not have any neighbours in $\bigcup_{j=2}^m \G_{\ell_j}$ (otherwise the distance from $y$ to $x$ will be less than $m$, a contradiction). 
Now, since the distance-$i$ partition is equitable, there does not exist $z\in\G_i^\ra(x)$ which has a neighbour in $\bigcup_{j=2}^m \G_{\ell_j}$ (otherwise $y$ would have a neighbour in $\bigcup_{j=2}^m \G_{\ell_j}$, a contradiction). In the end, since $\partial(y,w_1)=1$, $y\in\G_i^\ra(x)$, $w_1\in\G^\ra_{\ell_1}(x)$, and having equitable distance-$i$ partition, we can conclude that every vertex $z\in\G_i^\ra(x)$ has exactly $d_{i\ell_1}^\ra>0$ neighbours in $\G_{\ell_1}^\ra(x)$. We got that the distance from an arbitrary $z\in\G_i^\ra(x)$ to $x$ is equal to $m$. The result follows.
\end{proof}

\begin{proposition}
\label{dF}
Let $\G$ denote a distance-regular digraph. For any $x\in X$, the partition 
$$
\{\G_0^{\la}(x),\G_1^{\la}(x),\ldots,\G_D^{\la}(x)\}
$$ is equitable. Moreover,
$$
\{\G_0^{\ra}(x),\G_1^{\ra}(x),\ldots,\G_D^{\ra}(x)\}=
\{\G_0^{\la}(x),\G_1^{\la}(x),\ldots,\G_D^{\la}(x)\}.
$$
\end{proposition}

\begin{proof}
Our proof is by construction. Note that $\G_0^{\ra}(x)=\G_0^\la(x)$.

Since every vertex has the same eccentricity $D$, there exists $y_1\in X$ such that $\partial(y_1,x)=1$. If $y_1\in\G_{i_1}^\ra(x)$, by Lemma~\ref{dE}, for every $z\in\G_{i_1}^{\ra}(x)$, $\partial(z,x)=1$. This yields $\G_{i_1}^\ra(x)\subseteq\G_1^\la(x)$. 
Next, pick $y_2\in X$ such that $\partial(y_2,x)=2$. If $y_2\in\G_{i_2}^\ra(x)$, by Lemma~\ref{dE}, for every $z\in\G_{i_2}^{\ra}(x)$, $\partial(z,x)=2$. This yields $\G_{i_2}^\ra(x)\subseteq\G_2^\la(x)$. 

We continue in this way next $D-1$ steps, till when we end up with a vertex $y_D\in X$ such that $\partial(y_D,x)=D$ (since every vertex has same eccentricity, such a vertex exists). If $y_D\in\G_{i_D}^\ra(x)$, by Lemma~\ref{dE}, for every $z\in\G_{i_D}^{\ra}(x)$, $\partial(z,x)=D$. This yields $\G_{i_D}^\ra(x)\subseteq\G_D^\la(x)$. 

By construction, the elements of the set $\{\G_{i_\ell}^{\ra}(x)\}_{\ell=0}^D$ are disjoint, as well as the elements of the set $\{\G_i^{\la}(x)\}_{i=0}^D$. We also have 
$$
X
=\bigcup_{i=0}^D \G_i^{\ra}(x)
=\bigcup_{\ell=0}^D \G_{i_\ell}^{\ra}(x)
\subseteq 
\bigcup_{i=0}^D \G_i^{\la}(x) 
=X
$$ 
which yields 
$$
\G_{i_\ell}^\ra(x)=\G_\ell^\la(x)
\qquad(0\le\ell\le D).
$$
The result follows.
\end{proof}



\medskip
{\sc Damerell} in \cite{Drm} had proved the following result.

\begin{theorem}[{\cite[Theorems~1 and 2]{Drm}}]
\label{zk}
Let $\G$ denote a strongly connected digraph with vertex set $X$, diameter $D$ and girth $g$. Assume that there exists numbers $b_{ij}$ $(0\le i,j\le D)$ such that $|\G^\ra_1(y)\cap\G_j^\ra(x)|=b_{ij}$, for all $x\in X$, $y\in\G^\ra_i(x)$ $(0\le i,j\le D)$ (i.e., $\G$ is a distance-regular digraph in sense of {\sc Damerell} {\rm \cite{Drm}}). Then
$$
(A_i)^\top = A_{g-i}\qquad(1\le i\le g-1).
$$
Moreover, if $g\ge 3$, then either $D=g$ or $D=g-1$.
\end{theorem}

\begin{corollary}
\label{dH}
Let $\G$ denote a distance-regular digraph, and $\{A_i\}_{i=0}^D$ denote the distance-$i$ matrices. Then 
$$
A_i^\top\in\{A_0,A_1,\ldots,A_D\}\qquad(0\le i\le D).
$$
\end{corollary}

\begin{proof}
Immediately from the definition of distance-regular digraph and Theorem~\ref{zk}.
\end{proof}




\section{Characterization involving commutative association\break schemes}
\label{dA}

Let $\G$ denote a simple strongly connected digraph with vertex set $X$, diameter $D$, and let $\{A_0,A:=A_1,A_2,\ldots,A_D\}$ denote the distance-$i$ matrices of $\G$. Let $\A$ denote adjacency algebra of $\G$. Let $R_i=\{(x,y)\in X\times X\mid (A_i)_{xy}=1\}$ $(0\le i\le D)$ denote a subset of $X\times X$ (note that $\{R_i\}_{i=0}^D$ is a partition of $X\times X$). Let $\M$ denote the Bose--Mesner algebra of a $d$-class $P$-polynomial commutative association scheme $\XXi$ with standard basis $\{B_i\}_{i=0}^d$. In this section we prove two characterizations of a distance-regular digraph:
\begin{enumerate}[left=1cm, label={\sc(\Alph*)}]
\item $\G$ is a distance-regular digraph if and only if $(X,\{R_i\}_{i=0}^D)$ is a commutative $D$-class association scheme.
\item $\G$ is a distance-regular digraph if and only if $\A$ is the Bose--Mesner algebra of a commutative $D$-class association scheme.
\end{enumerate}


\begin{lemma}
\label{dI}
Let $\G=\G(A)$ denote a distance-regular digraph, and $\{A_0,A:=A_1,A_2,\ldots,A_D\}$ denote the distance-$i$ matrices. There exist complex scalars $p^h_{1i}$ $(0\le i,h\le D)$ such that 
$$
A_iA=\sum_{h=0}^D p^h_{i1} A_h.
$$
\end{lemma}

\begin{proof}
Note that
$$
(A_iA)_{xy}=|\G_i^\ra(x)\cap\G_1^\la(y)|\qquad(x,y\in X).
$$
Since $\G$ is a distance-regular digraph, by definition the distance-$i$ partition $\{\G^\ra_i(x)\}_{i=0}^D$ around a vertex $x$ has the same structure for every $x\in X$ with the same corresponding parameters. Then, for $\partial(x,y)=j$ we have
$$
|\G_i^\ra(x)\cap\G_1^\la(y)|=d^\la_{ji}\qquad(\mbox{does not depend on the choice of } y\in\G_j^\ra(x)).
$$
Now, if we define $p^j_{i1}:=d^\la_{ji}$ $(0\le i,j\le D)$, we have
$$
A_iA=\sum_{h=0}^D p^h_{i1} A_h.
$$
The result follows.
\end{proof}

\begin{proposition}
\label{jE}
Let $\G=\G(A)$ denote a distance-regular digraph, and $\{A_0,A:=A_1,A_2,\ldots,A_D\}$ denote the distance-$i$ matrices.
\begin{enumerate}[label=\rm(\roman*)]
\item There exist polynomials $p_i$ with $\deg p_i\le i$ $(0\le i\le D)$ such that $A_i=p_i(A)$.
\item For every $i,j$ $(0\le i,j\le D)$ there exist complex scalars $p^h_{ij}$ $(0\le h\le D)$ such that 
$$
A_iA_j=\sum_{h=0}^D p^h_{ij} A_h.
$$ 
Moreover, for the scalars $p^h_{ij}$ $(0\le i,j,h\le D)$ the following hold.
\begin{enumerate}[label=\rm(\roman*)]
\item $p^h_{ij}=0$, if $h>i+j$.
\item $p^h_{ij}\ne 0$, if $h=i+j$.
\end{enumerate}
\label{cB}
\item $A_iA_j=A_jA_i$ for every $i,j$ $(0\le i,j\le d)$ (i.e., for the {\em intersection numbers} $p^h_{ij}$ $(0\le i,j,h\le D)$ from \ref{cB} we have that $p^h_{ij}=p^h_{ji}$).\label{cF}
\end{enumerate}
\end{proposition}

\begin{proof}
Immediate from Lemma~\ref{dI}, Proposition~\ref{Cg}, and Corollary~\ref{eF}.
\end{proof}

\medskip
Note that our first characterization is very similar to the result of {\sc Damerell} given in Lemma~\ref{PI}.

\begin{theorem}[Characterization A]
\label{dJ}
Let $\G$ denote a simple strongly connected digraph with vertex set $X$, diameter $D$, and let $\{A_0,A:=A_1,A_2,\ldots,A_D\}$ denote the distance-$i$ matrices of $\G=\G(A)$. Let $\{R_i\}_{i=0}^D$ denote a partition of $X\times X$ where $R_i=\{(x,y)\in X\times X\mid (A_i)_{xy}=1\}$ $(0\le i\le D)$. Then, $\G$ is a distance-regular digraph if and only if $(X,\{R_i\}_{i=0}^D)$ is a commutative association scheme.
\end{theorem}

\begin{proof}
$(\La)$ Assume that $(X,\{R_i\}_{i=0}^D)$ is a commutative association scheme and consider a graph $\G=\G(A)$. Since the diameter of $\G$ is $D$ and we have a $D$-class association scheme, by Theorem~\ref{pi}, for every vertex $x\in X$, the distance-$i$ partition of $\G$ is equitable and the corresponding parameters do not depend on the choice of a vertex. In other words, for every $x\in X$, the distance-$i$ partition $\{\G_i^\ra(x)\}_{i=0}^D$ is equitable and of its structure do not depend on the choice of a vertex. The result follows. 

$(\Ra)$ Assume that $\G$ is a distance-regular digraph. From definition of distance-$i$ matrices it follows that (AS1) $A_0=I$ and (AS2) $\sum_{i=0}^D A_i= J$ hold. By Corollary~\ref{dH}, (AS3) $A_i^\top\in\{A_0,A_1,\ldots,A_D\}$ $(0\le i\le D)$ also holds. In the end, Proposition~\ref{jE} implies (AS4) $A_iA_j=\sum_{h=0}^D p^h_{ij} A_h$ for every $i,j$  $(0\le i,j\le D)$ and some complex scalars $p^h_{ij}$ $(0\le h\le D)$; and (AS5) $A_iA_j=A_jA_i$. Thus $(X,\{R_i\}_{i=0}^D)$ is a commutative $D$-class association scheme.
\end{proof}

\begin{corollary}
\label{oD}
Let $\G$ denote a digraph with vertex set $X$, diameter $D$, and let $A$ denote the adjacency matrix of $\G$. If $\G$ is a distance-regular digraph, then $\G$ has exactly $D+1$ distinct eigenvalues.
\end{corollary}

\begin{proof}
Let $\{A_0,A:=A_1,A_2,\ldots,A_D\}$ denote the distance-$i$ matrices of $\G$. By Theorem~\ref{dJ} (Characterization~A), the set $\{A_0,A_1,\ldots,A_D\}$ is a basis of some commutative $D$-class association scheme $\XXi$. By Proposition~\ref{jE} every distance-$i$ matrix $A_i$ $(0\le i\le D)$ can be written as polynomial in $A$, and with that $\XXi$ is generated by $A$. The result now follows from \cite[Corollary~3.5]{MPc}.
\end{proof}

\begin{proposition}
\label{oC}
Let $\M$ denote a $d$-class $P$-polynomial association scheme $\XXi$ with the standard basis $\{B_0,B_1,\ldots,B_d\}$, and assume that $(I=B_0,B=B_1,\ldots,B_d)$ is a $P$-polynomial ordering of $\XXi$. Let $\G=\G(B)$ denote a  digraph with the adjacency matrix $B$. Then, $\G$ is a graph of diameter $d$ and the $B_i$'s are the distance-$i$ matrices of $\G$.
\end{proposition}

\begin{proof}
Since $\XXi$ is a $P$-polynomial scheme, there is at least one ordering $(B_0,B=B_1,\ldots,B_d)$ of the standard basis of $\M$ such that there exists a  sequence of polynomials $\{p_i(t)\}_{i=0}^d$ with the property that $\deg(p_i)=i$ $(0\le i\le d)$ and $B_i=p_i(B)$. Note that $B$ is adjacency matrix of some (directed) graph $\G$. Since $J=\sum_{i=0}^d p_i(B)=q(B)$, $\G$ is strongly connected graph (see Theorem~\ref{Ra}). Taking $j=1$ in the property (AS4) of the standard basis, we have that there exist scalars $p^h_{i1}$ $(0\le i,h\le d)$ such that
$$
B_iB=\sum_{h=0}^d p^{h}_{i1} B_h\qquad(0\le i\le d).
$$
Similar as in \cite[Proposition~2.7.1]{BCN}, note that 
$$
\underbrace{p_{i}(t)p_1(t)}_{\deg = i+1}=
\sum_{h=0}^d p^{h}_{i1} \underbrace{p_h(t)}_{\deg = h}\qquad(0\le i\le d-1).
$$
This yields that if $h>i+1$ then $p^h_{i1}=0$; and that 
\begin{equation}
\label{jG}
p^{\ell}_{\ell-1,1}\ne 0\qquad (1\le \ell\le d). 
\end{equation}
With that we have
\begin{equation}
\label{oA}
B_iB=\sum_{h=0}^{i+1} p^{h}_{i1} B_h\qquad(0\le i\le d-1).
\end{equation}
Note also that, for any $y,z\in X$,
$$
(B_iB)_{zy}=|\G^{\ra}_i(z)\cap\G_1^{\la}(y)|\qquad(0\le i\le d-1),
$$
and that the scalars $p^h_{i1}$ $(0\le h,i\le d)$ are non-negative integers (since $\{B_i\}_{i=0}^d$ is the standards basis). Using \eqref{oA} and induction on $i$, in the next few lines we prove that $B_i$ $(0\le i\le d)$ is the distance-$i$ matrix of the digraph $\G=\G(B)$.

\bigskip
\noindent
{\sc Basis of induction.} $B_0$ is the identity matrix, while $B$ is the adjacency matrix of $\G$ by construction. The claim is true for $i\in\{0,1\}$.

\smallskip
\noindent
{\sc Step of induction.} Assume that
$$
(B_h)_{xy}=1 
\qquad\mbox{ if and only if }\qquad
\partial(x,y)=h
$$
holds for every $h\in\{0,1,\ldots,m-1\}$ for $m\ge 2$ $(m\le d)$. We use this assumption and prove that $(B_m)_{x,y}=1$ if and only if $\partial(x,y)=m$. From \eqref{oA} we have
\begin{equation}
\label{oB}
B_{m-1}B=
p^0_{m-1,1} B_0+
p^1_{m-1,1} B_1+\cdots+
p^m_{m-1,1} B_m.
\end{equation}

$(\La)$ Assume that $\partial(x,y)=m$. By induction assumption, we have $(B_i)_{xy}=0$ for $0\le i\le m-1$. By \eqref{oB} this implies
$$
\underbrace{(B_{m-1}B)_{xy}}_{=|\G^{\ra}_i(x)\cap\G^{\la}(y)|}=p^m_{m-1,1} (B_m)_{xy}.
$$
Since $\partial(x,y)=m$, we have $|\G^{\ra}_{m-1}(x)\cap\G_1^{\la}(y)|\ne 0$, and consequently $(B_m)_{xy}=1$.

$(\Ra)$ Assume that $(B_m)_{xy}=1$. Since $\{B_i\}_{i=0}^d$ is the standard basis, from \eqref{oB} we have
$$
\underbrace{(B_{m-1}B)_{xy}}_{=|\G^{\ra}_i(x)\cap\G_1^{\la}(y)|}=p^m_{m-1,1}.
$$
By \eqref{jG}, $p^m_{m-1,1}\ne 0$, which yields $|\G^{\ra}_i(x)\cap\G^{\la}(y)|\ne 0$. Therefore, $\partial(x,y)=m$.

\bigskip
The result follows.
\end{proof}

\begin{proposition}
\label{Oc}
Let $\M$ denote the Bose--Mesner algebra of a $P$-polynomial $d$-class association scheme $\XXi$ with standard basis $\{B_i\}_{i=0}^d$. If $(B_0,B:=B_1,\ldots,B_d)$ is a $P$-polynomial ordering for $\XXi$, then $\G=\G(B)$ is a distance-regular digraph.
\end{proposition}

\begin{proof}
By Proposition~\ref{oC}, for the $P$-polynomial ordering $(B_0,B:=B_1,\ldots,B_d)$ we have that the $B_i$'s are the distance-$i$ matrices of the graph $\G=\G(B)$. This implies that diameter $D$ of $\G$ is equal to $d$. Let $\{R_i\}_{i=0}^D$ denotes a partition of $X\times X$ where $R_i=\{(x,y)\in X\times X\mid (B_i)_{xy}=1\}$ $(0\le i\le D)$. Then, $(X,\{R_i\}_{i=0}^D)$ is a commutative association scheme (see Proposition~\ref{Cg}(iv)) and the result follows from Theorem~\ref{dJ} (Characterization~A).
\end{proof}

\begin{theorem}[Characterization B]
\label{eB}
Let $\G=\G(A)$ denote a strongly connected digraph with vertex set $X$, diameter $D$, and let $\A$ denote the adjacency algebra of $\G$. Then, $\G$ is a distance-regular digraph if and only if $\A$ is the Bose--Mesner algebra of a commutative $D$-class association scheme.
\end{theorem}

\begin{proof}$(\La)$ Assume that $\A$ is the Bose--Mesner algebra of a $D$-class commutative association scheme $\XXi$. Let $\{B_i\}_{i=0}^D$ denote the standard basis of $\A$ and $A$ denote the adjacency matrix of $\G$. Since $A$ generates $\M$, $\G$ has $D+1$ distinct eigenvalues (see, for example, \cite[Corollary~3.5]{MPc}). By Proposition~\ref{Ha}(ii), every distance-$i$ matrix $A_i$ $(0\le i\le D)$ of $\G$ belongs to $\A$. Since $\{A_i\}_{i=0}^D$ is a linearly independent set, $\A=\Span\{A^0,A^1,\ldots,A^D\}=\Span\{A_0,A_1,\ldots,A_D\}$. Moreover, the set $\{A_0,A_1,\ldots,A_D\}$ is equal to the standard basis $\{B_i\}_{i=0}^D$, i.e., we can permute indices in such a way that $A_i=B_i$ $(i\le i\le D)$. The result now follows from Theorem~\ref{dJ} (Characterization~A).

$(\Ra)$ Assume that $\G$ is a distance-regular digraph. The result follows immediately from Proposition~\ref{jE}(i) and Theorem~\ref{dJ} (Characterization~A).
\end{proof}



\section{Characterization involving the adjacency matrix and the distance-$\boldsymbol{i}$ matrices of a graph}
\label{eA}

In this section, we prove that a graph $\G$ satisfies properties (i), (ii) of Problem~\ref{Aa} if and only if $\G$ is a distance-regular digraph (see Theorem~\ref{eD}).

\begin{theorem}[Characterization~C]
\label{eD}
Let $\G$ denote a simple strongly connected digraph with vertex set $X$,  diameter $D$, and let $\{A_0,A:=A_1,A_2,\ldots,A_D\}$ denote the distance-$i$ matrices of $\G$. Then, $\G$ is a distance-regular digraph if and only if $A^\top\in\{A_0,A_1,\ldots,A_D\}$ and any of the following holds.
\begin{enumerate}[label=\rm(\roman*)]
\item $A$ acts by right multiplication as a linear operator on the vector space $\Span\{A_0,A_1,\ldots,A_D\}$.
\item There exist complex scalars $p^h_{i1}$ $(0\le i,h\le D)$ such that $A_iA=\sum_{h=0}^D p^h_{i1} A_h$ $(0\le i\le D)$.
\end{enumerate}
\end{theorem}

\begin{proof} (i) We prove that $\G$ is a distance-regular digraph if and only if $A^\top\in\{A_0,A_1,\ldots,A_D\}$ and $A$ acts by right multiplication as a linear operator on the vector space $\Span\{A_0,A_1,\ldots,A_D\}$.

$(\La)$ Assume that $A^\top\in\{A_0,A_1,\ldots,A_D\}$ and $A$ acts by right multiplication as a linear operator on the vector space $\Span\{A_0,A_1,\ldots,A_D\}$. This yields that there exist complex scalars $p^h_{i1}$ $(0\le i,h\le D)$ such that 
$$
A_iA=\sum_{h=0}^D p^h_{i1} A_h
\qquad(0\le i\le D).
$$
In the next few lines, we prove that the algebra $\M=\Span\{A_0,A_1,\ldots,A_D\}$ generated by the distance-$i$ matrices is the Bose-Mesner algebra of a $D$-class $P$-polynomial commutative association schemes. By definition of distance-$i$ matrices, we have (AS1) $A_0=I$ and (AS2) $\sum_{i=0}^D A_i=J$. By Proposition~\ref{eL}, property (AS3) $A_i^\top\in\{A_0,A_1,\ldots,A_D\}$ holds. By Corollary~\ref{eF}, (AS4) the vector space $\M$ is closed with respect to the ordinary matrix multiplication. In the end, by Proposition~\ref{Cg}, we have two things: (AS5) $A_iA_j=A_jA_i$ $(0\le i,j\le D)$; and the ordering $(A_0,A=A_1,\ldots,A_D)$ is a $P$-polynomial ordering. The result now follows from Theorem~\ref{eB} (Characterization~B).

$(\Ra)$ Assume that $\G$ is a distance-regular digraph. The result follows immediately from Corollary~\ref{dH} and Lemma~\ref{dI}.

(ii) Immediate from (i).
\end{proof}

\begin{remark}{\rm
In Theorem~\ref{eD}, we have restriction that $A$ acts by right multiplication as a linear operator on the vector space $\Span\{A_0,A_1,\ldots,A_D\}$. The claim remains true if we replace right multiplication by left multiplication, i.e., 
$\G$ is a distance-regular digraph if and only if $A^\top\in\{A_0,A_1,\ldots,A_D\}$ and $A$ acts by left multiplication as a linear operator on the vector space $\Span\{A_0,A_1,\ldots,A_D\}$ (see Section~\ref{na}). 
}\end{remark}

Our Characterizations C, C', C'' and D (see Theorems~\ref{eD}, \ref{eQ}, \ref{eP} and \ref{fB}) are in some sense very similar to result of {\sc Comellas et al.} from \cite{CfFgM} (see Subsection~\ref{zm}).

\begin{theorem}[Characterization~C']
\label{eQ}
Let $\G$ denote a digraph with vertex set $X$, diameter $D$, and let $\{A_0,A:=A_1,A_2,\ldots,A_D\}$ denote the distance-$i$ matrices of $\G=\G(A)$. Then, $\G$ is a distance-regular digraph if and only if $A^\top\in\{A_0,A_1,\ldots,A_D\}$ and $\{A_0,A_1,\ldots,A_D\}$ is a basis of $\A$.
\end{theorem}

\begin{proof}
$(\La)$ Assume that $A^\top\in\{A_0,A_1,\ldots,A_D\}$ and $\{A_0,A_1,\ldots,A_D\}$ is a basis of $\A$. Then $A_iA\in\Span\{A_0,A_1,\ldots,A_D\}$, and the result follows from Theorem~\ref{eD} (Characterization~C).

$(\Ra)$ Assume that $\G=\G(A)$ is a distance-regular digraph. By Corollary~\ref{dH}, $A^\top\in\{A_0,A=A_1,\ldots,A_D\}$. By Lemma~\ref{dI}, $A_iA\in\Span\{A_0,A_1,\ldots,A_D\}$ which yields $\A\subseteq\Span\{A_0,A_1,\ldots,A_D\}$. From Corollary~\ref{oD}, $A$ has $D+1$ distinct eigenvalues. The result follows.
\end{proof}

\begin{theorem}[Characterization~C'']
\label{eP}
Let $\G$ denote a simple strongly connected digraph with vertex set $X$,  diameter $D$, and let $\{A_0,A:=A_1,A_2,\ldots,A_D\}$ denote the distance-$i$ matrices of $\G=\G(A)$. Then, $\G$ is a distance-regular digraph if and only if at least one of the following hold.
\begin{enumerate}[label=\rm(\roman*)]
\item $A^\top\in\{A_0,A_1,\ldots,A_D\}$, and there exist complex scalars $p^h_{ij}$ $(0\le i,j,h\le D)$ such that $A_iA_j=\sum_{h=0}^D p^h_{ij} A_h$ $(0\le i,j\le D)$. 
\item $A_i^\top\in\{A_0,A_1,\ldots,A_D\}$ $(0\le i\le D)$, and $A_iA_j=\sum_{h=0}^D p^h_{ij} A_h$ $(0\le i,j\le D)$ holds for some scalars $p^h_{ij}$ $(0\le i,j,h\le D)$.
\item $A^\top\in\{A_0,A_1,\ldots,A_D\}$ and there exist complex scalars $p^h_{ij}$ $(0\le i,j,h\le D)$ such that $|\G_i^\ra(x)\cap\G_j^\la(y)|=p^h_{ij}$, for any $x\in X$, $y\in\G_h^\ra(x)$ $(0\le i,j,h\le D)$.
\end{enumerate}
\end{theorem}

\begin{proof}
(i), (ii). Immediate from Propositions~\ref{Cg} and \ref{eL}, Theorem~\ref{eB} (Characterization~B), and Theorem~\ref{eQ} (Characterization~C').

(iii) Immediate from (i).
\end{proof}

\section{Characterization involving distance-$\boldsymbol{i}$ polynomials}
\label{fA}

Let $\G=\G(A)$ denote a strongly connected (di)graph with adjacency matrix $A$, and let $\{A_0,A:=A_1,A_2,\ldots,A_D\}$ denote the distance-$i$ matrices of $\G$. Polynomials $\{p_i\}_{i=0}^D$ which have the property that $A_i=p_i(A)$ $(0\le i\le D)$ are called {\em distance-$i$ polynomials} (note that $p_0(t)=1$ and $p_1(t)=t$). In this section, we show that $\G$ is a distance-regular digraph if and only if $A^\top\in\{A_0,A_1,\ldots,A_D\}$ and there exists a sequence of distance-$i$ polynomials $(p_i)_{0\le i\le D}$ with the property that $\dgr p_i\le i$ $(0\le i\le D)$.

As we already mention in Section~\ref{eA} our Characterization~D (Theorems~\ref{fB}) is in some sense very similar to the result of {\sc Comellas et al.} from \cite{CfFgM} (see Subsection~\ref{zm}).

\begin{theorem}[Characterization D]
\label{fB}
Let $\G$ denote a simple strongly connected digraph with vertex set $X$,  diameter $D$, and let $\{A_0,A:=A_1,A_2,\ldots,A_D\}$ denote the distance-$i$ matrices of $\G=\G(A)$. Then $\G$ is a distance-regular digraph if and only if $A^\top\in\{A_0,A_1,\ldots,A_D\}$ and there exists a sequence of distance-$i$ polynomials $(p_i)_{0\le i\le D}$ such that $\dgr p_i\le i$.
\end{theorem}

\begin{proof}
$(\La)$ Assume that $A^\top\in\{A_0,A_1,\ldots,A_D\}$ and that there exist polynomials $p_i(t)$ $(0\le i\le D)$ with $\dgr p_i\le i$ such that $A_i=p_i(A)$. By Proposition~\ref{zi}, the set of distance-$i$ matrices $\{A_i\}_{i=0}^D$ is a basis of the adjacency algebra $\A$ of $\G$, and with it $d=D$. Now it is routine to show that $\A$ is the Bose--Mesner algebra of a $P$-polynomial commutative association scheme (see Section~\ref{Av}). The result follows from Theorem~\ref{eB} (Characterization~B).




$(\Ra)$ Assume that $\G=\G(A)$ is a distance-regular digraph. The result follows immediately from Lemma~\ref{dI}, Proposition~\ref{Cg} and Corollary~\ref{dH}.
\end{proof}

\section{Characterization involving walks}
\label{hA}

In this section, we show that a strongly connected digraph $\G$ is a distance-regular digraph if and only if $A^\top\in\{A_0,A_1,\ldots,A_D\}$ and for each non-negative integer $\ell$, the number of walks of length $\ell$ in $\G$ from a vertices $x$ to a vertex $y\in\G_h^{\ra}(x)$ only depends on $h$ (on the distance from $x$ to $y$).

Our Characterization~E (see Theorem~\ref{fE}) is in some sense very similar to the result of {\sc Comellas et al.} from \cite{CfFgM} (see Subsection~\ref{zm}).

\begin{theorem}[Characterization~E]
\label{fE}
Let $\G$ denote a digraph with vertex set $X$, diameter $D$, and let $\{A_0,A:=A_1,A_2,\ldots,A_D\}$ denote the distance-$i$ matrices of $\G=\G(A)$. Then, $\G$ is a distance-regular digraph if and only if $A^\top\in\{A_0,A_1,\ldots,A_D\}$ and for each non-negative integer $\ell$ $(0\le\ell\le D)$, the number of walks of length $\ell$ in $\G$ between two vertices $x,y\in X$ only depends on $h=\partial(x,y)$ (distance from $x$ to $y$).
\end{theorem}

\begin{proof}
$(\La)$ Assume that $A^\top\in\{A_0,A_1,\ldots,A_D\}$ and that, for each non-negative integer $\ell$, the number of walks of length $\ell$ in $\G$ between two vertices $x,y\in X$ only depends on $h=\partial(x,y)$ (distance from $x$ to $y$). By Proposition~\ref{zi}, the set of distance-$i$ matrices $\{A_i\}_{i=0}^D$ is a basis of the adjacency algebra $\A$ of $\G$, and with it we have
$A_iA\in\Span\{A_0,A_1,\ldots,A_D\}$. The result follows from Theorem~\ref{eD} (Characterization~C).

$(\Ra)$ Assume that $\G$ is a distance-regular digraph. The result is not hard to deduce from Theorem~\ref{eP} (Characterization~C'') and Lemma~\ref{hB}.
\end{proof}


\section{Distance-regularity with respect to the partition $\boldsymbol{\{\G_i^\la(x)\}_{i=0}^D}$}
\label{na}

Let $\G$ denote a simple strongly connected digraph with vertex set $X$ and diameter $D$. In this section, we show that $\G$ is a distance-regular digraph if and only if for every vertex $x\in X$, the partition $\{\G^\la_0(x),\G^\la_1(x),\ldots,\G^\la_D(x)\}$ of $X$ is equitable and its corresponding parameters do not depend on the choice of $x$.

\begin{proposition}
\label{nb}
Let $\G$ denote a simple strongly connected digraph with vertex set $X$,  diameter $D$, and let $\{A_0,A:=A_1,A_2,\ldots,A_D\}$ denote the distance-$i$ matrices of $\G$. If there exist complex scalars $r^h_{i1}$ $(0\le i,j,h\le D)$ such that 
\begin{equation}
\label{nc}
AA_i=\sum_{h=0}^D r^h_{1i} A_h
\qquad(0\le i,j\le D),
\end{equation}
then the following hold.
\begin{enumerate}[label=\rm(\roman*)]
\item For the scalars $r^j_{1i}$ $(0\le i,j\le D)$, $r^j_{1i}=0$ if $j>i+1$; and $r^j_{1i}\ne 0$ if $j=i+1$.
\item There exist polynomials $p_i$ with $\deg p_i\le i$ $(0\le i\le D)$ such that $A_i=p_i(A)$.
\item $A_iA_j=A_jA_i$ $(0\le i,j\le d)$.
\end{enumerate}
\end{proposition}

\begin{proof}
The proof is similar to the proof of Proposition~\ref{Cg}.

(i) Pick $x\in X$. For any $y\in\G_j^\la(x)$, taking the $yx$-entry of the left side of \eqref{nc}, we have
$$
|\G_i^\la(x)\cap \G_1^\ra(y)|=r^j_{1i}.
$$
By considering the partition $\{\G_i^\la(x)\}_{i=0}^D$, note that $|\G_i^\la(x)\cap \G_1^\ra(y)|\ne 0$ if and only if $y\in\bigcup_{h=0}^{i+1}\G_h^{\la}(x)$, which yields $r^j_{1i}=0$ if $j>i+1$. On the other hand, if $y\in\G^\la_{i+1}(x)$ then $|\G_i^\la(x)\cap \G_1^\ra(y)|\ne 0$, which yields $r^{i+1}_{1i}\ne 0$. 

(ii), (iii) By (i), \eqref{nc} can be written as
$$
AA_i=\sum_{\ell=0}^{i+1} r^\ell_{1i} A_\ell \qquad(0\le i\le D).
$$
The rest of the proof is routine (see the proof of Proposition~\ref{Cg}(iii),(iv)).
\end{proof}

\begin{proposition}
\label{nd}
Let $\G$ denote a simple strongly connected digraph with vertex set $X$,  diameter $D$, and let $\{A_0,A:=A_1,A_2,\ldots,A_D\}$ denote the distance-$i$ matrices of $\G$. Assume that there exist complex scalars $r^h_{1h}$ $(0\le i,h\le D)$ such that 
\begin{equation}
AA_i=\sum_{h=0}^D r^h_{1i} A_h
\qquad(0\le i\le D),
\end{equation}
and that $A^\top\in\{A_0,A_1,\ldots,A_D\}$. Then, the following hold.
\begin{enumerate}[label=\rm(\roman*)]
\item The vector space $\Span\{A_0,A_1,\ldots,A_D\}$ is closed with respect to the ordinary matrix multiplication.
\item $A$ has $D+1$ distinct eigenvalues.
\item $A_i^\top\in\{A_0,A_1,\ldots,A_D\}$ $(0\le i\le D)$.
\end{enumerate}
\end{proposition}

\begin{proof}
Routine (use Proposition~\ref{nb}; see the proofs of Corollary~\ref{eF}(i) and Proposition~\ref{eL}).
\end{proof}

\begin{theorem}[Characterization F]
\label{ne}
Let $\G$ denote a simple strongly connected digraph with vertex set $X$ and diameter $D$. Then, $\G$ is a distance-regular digraph if and only if for every vertex $x\in X$, the partition $\{\G^\la_0(x),\G^\la_1(x),\ldots,\G^\la_D(x)\}$ of $X$ is equitable and its corresponding parameters do not depend on the choice of $x$.
\end{theorem}

\begin{proof}
$(\La)$ Assume that for every vertex $x\in X$, the partition $\{\G^\la_0(x),\G^\la_1(x),\ldots,\G^\la_D(x)\}$ of $X$ is equitable and its corresponding parameters do not depend on the choice of $x$. This yields that for any $x,y\in X$ there exist integers $r^j_{1i}$ $(0\le i,j\le D)$ such that 
$$
|\G_i^\la(x)\cap \G_1^\ra(y)|=r^j_{1i},\quad \mbox{for all }y\in\G_j^\la(x)\qquad (0\le i,j\le D).
$$
Since $(AA_i)_{yx}=|\G_i^\la(x)\cap\G^\ra_1(y)|$, we also have
$$
AA_i=\sum_{h=0}^{D} r^h_{1i} A_{\ell_h}.
$$
On the other hand, again, since the partition $\{\G^\la_0(x),\G^\la_1(x),\ldots,\G^\la_D(x)\}$ of $X$ is equitable and its corresponding parameters do not depend on the choice of $x$, for any $x,y\in X$ there exist integers $s^j_{1i}$ $(0\le i,j\le D)$ such that
$$
|\G_i^\la(x)\cap \G_1^\la(y)|=s^j_{1i},\qquad \mbox{for all }y\in\G_j^\la(x).
$$
Since $(A^\top A_i)_{yx}=|\G_i^\la(x)\cap\G^\la_1(y)|$, we also have
$$
A^\top A_i=\sum_{h=0}^{D} s^h_{1i} A_{\ell_h},
$$
which yields $A^\top\in\{A_0,A_1,\ldots,A_D\}$ (as $A_0=I$). Now, from Proposition~\ref{nd} it is not hard to see that $\A$ is the Bose--Mesner algebra of a commutative $D$-class association scheme, and the result follows from Theorem~\ref{eB} (Characterisation~B).

$(\Ra)$ Assume that $\G$ is a distance-regular digraph. The result follows immediately from Proposition~\ref{dF}.
\end{proof}

\begin{theorem}[Characterization F']
\label{nf}
Let $\G$ denote a simple strongly connected digraph with vertex set $X$,  diameter $D$, and let $\{A_0,A:=A_1,A_2,\ldots,A_D\}$ denote distance-$i$ matrices of $\G$. Then, $\G$ is a distance-regular digraph if and only if $A^\top\in\{A_0,A_1,\ldots,A_D\}$ and there exist complex scalars $r^h_{1i}$ $(0\le i,h\le D)$ such that $AA_i=\sum_{h=0}^D r^h_{i1} A_h$ $(0\le i\le D)$.
\end{theorem}

\begin{proof}
Routine.
\end{proof}


\section{Distance-regularity in sense of {\sc Damerell}}
\label{iA}

A digraph $\G$ with diameter $D$ is distance-regular in sense of {\sc Damerell} \cite{Drm} if, for any pair of vertices $x,y\in X$ such that $y\in\G_i^\ra(x)$ $(0\le i\le D)$, the numbers
$$
b_{ij} := |\G_j^{\ra}(x)\cap\G_1^\ra(y)|,
$$ 
for each $j$ ($0\le j\le D$), do not depend on the chosen vertices $x$ and $y$, but only on the distance $i$ from $x$ to $y$. In this section, we show that $\G$ is distance-regular graph if and only if $\G$ is distance-regular in sense of {\sc Damerell}.

\begin{theorem}[Characterization~G]
\label{iB}
Let $\G$ denote a simple strongly connected digraph with vertex set $X$ and diameter $D$. Then $\G$ is a distance-regular digraph if and only if for any pair of vertices $x,y\in X$ such that $y\in\G_h^\ra(x)$ $(0\le h\le D)$, the numbers
$$
s^h_{i1} := |\G_i^{\ra}(x)\cap\G_1^\ra(y)|,
$$ 
for each $i$ (such that $0\le i\le D$), do not depend on the chosen vertices $x$ and $y$, but only on the distance $h$ from $x$ to $y$.
\end{theorem}

\begin{proof}
Let $\{A_0,A:=A_1,A_2,\ldots,A_D\}$ denote the distance-$i$ matrices of $\G=\G(A)$.

$(\La)$ Assume that for any pair of vertices $x,y\in X$ such that $y\in\G_h^\ra(x)$ $(0\le h\le D)$, the numbers
$$
s^h_{i1} := |\G_i^{\ra}(x)\cap\G_1^\ra(y)|,
$$ 
for each $i$ (such that $0\le i\le D$; note that $s^h_{i1}=0$ if $i>h+1$), do not depend on the chosen vertices $x$ and $y$, but only on the distance $h$ from $x$ to $y$. For any $y\in\G_h^\ra(x)$ we have
\begin{align}
(A_iA^\top)_{xy} &= \sum_{z\in X} (A_i)_{xz}(A^\top)_{zy}\nonumber\\
&= \sum_{z\in X} (A_i)_{xz}(A)_{yz}\label{Xc}\\
&= |\G_i^{\ra}(x)\cap\G_1^\ra(y)|,\nonumber\\
(\sum_{\ell=0}^D s^\ell_{i1} A_\ell)_{xy} &= s^h_{i1}.\nonumber
\end{align}
Thus, for the complex scalars $s^\ell_{i1}$ $(0\le i,\ell\le D)$, we have
\begin{equation}
\label{Xa}
A_iA^\top=\sum_{\ell=0}^D s^\ell_{i1} A_\ell
\qquad(0\le i\le D).
\end{equation}
Since $A_0=I$, \eqref{Xa} implies $A^\top\in\{A_0,A_1,\ldots,A_D\}$.


By Corollary~\ref{dH} we have $A^\top_i\in\{A_0,A_1,\ldots,A_D\}$ $(0\le i\le D)$. 
Now, for arbitrary $x,y\in X$, consider the $yx$-entry of $AA_i$. For the moment let $A^\top_i=A_{k}$ and $m=\partial(y,x)$. We have
\begin{align*}
(AA_i)_{yx} 
&= \sum_{z\in X} (A)_{yz} (A_i)_{zx}\\
&= \sum_{z\in X} (A)_{yz} (A^\top_i)_{xz}\\
&= \sum_{z\in X} (A_k)_{xz}(A)_{yz}\\
&= |\G_k^{\ra}(x)\cap\G_1^\ra(y)|\\
(\sum_{\ell=0}^D s^\ell_{k1} A_\ell)_{yx} &= s^m_{k1},
\end{align*}
which yields 
\begin{equation}
\label{ma}
AA_i=\sum_{\ell=0}^D s^\ell_{k1} A_\ell \qquad(0\le i\le D,\mbox{ and $k$ is unique index s.t. } A^\top_i=A_k).
\end{equation}
The above line \eqref{ma} can be written as
\begin{equation}
\label{mb}
AA_i=\sum_{\ell=0}^D r^\ell_{1i} A_\ell \qquad(0\le i\le D),
\end{equation}
for some complex scalars $r^\ell_{i1}$ $(0\le i,\le D)$. Using the results of Propositions~\ref{nb} and \ref{nd}, it is routine to show that $\{A_0,A_1,\ldots,A_D\}$ is a basis of $\A$. The result follows form Theorem~\ref{eQ} (Characterization~C').

$(\Ra)$ Note that \eqref{Xc} holds. The result follows from Theorem~\ref{eB} (Characterization~B).
\end{proof}

\begin{theorem}[Characterization~G']
\label{iC}
Let $\G$ denote a simple strongly connected digraph with vertex set $X$ and diameter $D$. Then, $\G$ is a distance-regular digraph if and only if there exist numbers $b_{ij}$ $(0\le i,j\le D)$ such that $|\G^\ra_1(y)\cap\G_j^\ra(x)|=b_{ij}$, for all $x\in X$, $y\in\G^\ra_i(x)$ $(0\le i,j\le D)$.
\end{theorem}

\begin{proof}
Immediate from Theorem~\ref{iB} (Characterization~G).
\end{proof}

\section{Weakly distance-regularity in sense of {\sc Wang} and {\sc Suzuki}}
\label{kA}

A digraph $\G$ with vertex set $X$ is weakly distance-regular in sense of {\sc Wang} and {\sc Suzuki} \cite{WS} if $(X,\{R_{\ii}\}_{\ii\in\Delta})$ is a $|\Delta|$-class association scheme, where $\Delta=\{(\partial(x,y),\partial(y,x))\mid x,y\in X\}$, and for any $\ii\in\Delta$, $R_{\ii}$ denote the set of ordered pairs $(x,y)\in X\times X$ such that $(\partial(x,y),\partial(y,x))=\ii$. In this section, we show which specific subfamily of weakly distance-regular digraphs in sense of {\sc Wang} and {\sc Suzuki} \cite{WS} is that of distance-regular digraphs.

\begin{theorem}[Characterization~H]
\label{kB}
Let $\G$ denote a simple strongly connected digraph with vertex set $X$ and diameter $D$. Let $\Delta=\{(\partial(x,y),\partial(y,x))\mid x,y\in X\}$, and  let $R_{\ii}$ denote the set of ordered pairs $(x,y)\in X\times X$ such that $(\partial(x,y),\partial(y,x))=\ii$, for any $\ii\in\Delta$. Then, $\G$ is a distance-regular digraph if and only if $(X,\{R_{\ii}\}_{\ii\in\Delta})$ is a commutative $D$-class association scheme.
\end{theorem}

\begin{proof} Let $\{A_0,A:=A_1,A_2,\ldots,A_D\}$ denote the distance-$i$ matrices of $\G=\G(A)$.

$(\La)$ Assume that $(X,\{R_{\ii}\}_{\ii\in\Delta})$ is a commutative $D$-class association scheme $\XXi$. This yields $|\Delta|=D+1$, and assume that $\Delta=\{\ii=(i,i^*)\mid 0\le i\le D\}$ (this assumption is valid, since diameter of graph is $D$). Note that for arbitrary $(z,y),(u,v)\in R_{\ii}$ $(0\le i\le D)$ we have $\partial(z,y)=\partial(u,v)=i$, as well as $\partial(y,z)=\partial(v,u)=i^*$ (note a similarity with Proposition~\ref{Ha}(i)). Let $\{B_0,B_1,\ldots,B_D\}$ denote the standard basis of the Bose--Mesner algebra of $\XXi$, i.e., let $(B_i)_{xy}=1$ $(0\le i\le D)$ if $(x,y)\in R_{\ii}$, and $(B_i)_{xy}=0$ otherwise. Note that then we have $A=B_1$, as well as $A_i=B_i$ $(0\le i\le D)$. The result follows from Theorem~\ref{dJ} (Characterization~A).

$(\Ra)$ Assume that $\G$ is a distance-regular digraph. By Corollary~\ref{dH},  for each $i$ $(0\le i\le D)$ there exists a unique $i^*$ $(0\le i^*\le D)$ such that $A_i^\top=A_{i^*}$. This implies that our set $\Delta$ is in fact $\Delta=\{(i,i^*)\mid 0\le i\le D\}$ and that $R_{\ii}=\{(x,y)\in X\times X\mid (A_i)_{xy}=1\}$ $(\ii\in\Delta)$. The result follows from Theorem~\ref{dJ} (Characterisation~A).
\end{proof}

With reference to Theorem~\ref{kB}, the case when $|\Delta|=d+1$ (where $d+1$ is  the number of distinct eigenvalues of $\G$) we described in \cite[Section~6]{MPc}.


\section{Case when $\boldsymbol{A_D}$ is polynomial in $\boldsymbol{A}$}
\label{oa}

In this section we prove that $\G=\G(A)$ is a distance-regular digraph if and only if $\G$ is a regular graph, has spectrally maximum diameter (i.e., $D=d$), and the matrices $A^\top$ and $A_D$ are polynomials in $A$.

\begin{definition}
\label{of}{\rm
Let $A\in\Mat_{X}(\CC)$ denote a normal matrix with $d+1$ distinct eigenvalues, such that $A\jj=\lambda_0\jj$.
The set of so-called {\em predistance polynomials} $\{p_0,p_1,\ldots,p_d\}$, is a set of orthogonal polynomials with respect to the inner product 
\begin{equation}
\label{oe}
\langle p, q\rangle={1\over |X|}\trace(p(A)\ol{q(A)}^\top),
\end{equation}
(defined on a ring of all polynomials $\mathbb{C}_d[t]=\{a_0+a_1t+\ldots+a_dt^d \mid a_i\in\mathbb{C},\,0\le i\le d\}$ of degree at most $d$ with coefficients in $\mathbb{C}$) such that $\deg(p_i)=i$ $(0\le i\le d)$ and normalized in a way that $\|p_i\|^2=p_i(\lambda_0)$ where $p(\lambda_0)>0$. (More about predistance polynomials reader can find in \cite[Section~4]{MoP}.)
}\end{definition}

\begin{lemma}
\label{oi}
With reference to {\rm Definition~\ref{of}}, let $\Gamma$ denote a simple connected regular (di)graph with adjacency matrix $A$, vertex set $X$, valency $k$ and assume that $A$ is a normal matrix with $d+1$ distinct eigenvalues. If $\{p_0,p_1,\ldots,p_d\}$ denote the set of the predistance polynomials, then
$$
\sum_{i=0}^d p_i(A) = J.
$$
\end{lemma}

\begin{proof}
Immediate from \cite[Lemma~4.5]{MoP}.
\end{proof}

\begin{proposition}
\label{ol}
With reference to {\rm Definition~\ref{of}}, let $\Gamma$ denote a simple connected $k$-regular (di)graph with adjacency matrix $A$.
Assume that that $A$ is a normal matrix with $d+1$ distinct eigenvalues and that $\G$ has spectrally maximum diameter (i.e. $D=d$). Let $\{p_0,p_1,\ldots,p_D\}$ denote the set of the predistance polynomials, and $A_D$ the distance-$D$ matrix. If there exists a polynomial $q(t)\in\RR_{D}[t]$ such that $A_D=q(A)$, then 
$$
q(t)=p_D(t).
$$
\end{proposition}

\begin{proof}
Immediate from \cite[Proposition~4.6]{MoP}.
\end{proof}

\begin{proposition}
\label{ok}
With reference to {\rm Definition~\ref{of}}, let $\Gamma$ denote a simple connected $k$-regular (di)graph with adjacency matrix $A$, adjacency algebra $\A$ and assume that $A$ is a normal matrix with $d+1$ distinct eigenvalues. Assume that $\G$ has spectrally maximum diameter (i.e. $D=d$). Let $\{p_0,p_1,\ldots,p_D\}$ denote the set of the predistance polynomials, and $A_D$ the distance-$D$ matrix. If $A^\top,A_D\in \A$, then 
$$
A_i=p_i(A)\qquad(0\le i\le D).
$$
\end{proposition}

\begin{proof}
Immediate from Lemma~\ref{on} and \cite[Proposition~5.3]{MoP}.
\end{proof}

Our main result of this section, Theorem~\ref{om}, is just ``directed'' version of the following result from theory of distance-regular (undirected) graphs.

\begin{proposition}[{\cite[Proposition~2]{Fsp} or \cite{FGJ}}]
An undirected regular graph $\G$ with diameter $D$ and $d + 1$ distinct eigenvalues is a distance-regular if and only if $D = d$ and the distance-$D$ matrix $A_D$ is a polynomial in $A$.
\end{proposition}

\begin{theorem}[Characterization~I]
\label{om}
Let $\G$ denote a simple strongly connected digraph with vertex set $X$, adjacency matrix $A$, diameter $D$, and let $A_D$ denote the distance-$D$ matrix. Then, $\G=\G(A)$ is a distance-regular digraph if and only if $\G$ is regular, has spectrally maximum diameter (i.e., $D=d$), and the matrices $A^\top$ and $A_D$ are polynomials in $A$.
\end{theorem}

\begin{proof}
$(\La)$ Assume that $\G$ is a regular graph, has spectrally maximum diameter $(D=d)$, and that $A^\top$ and $A_D$ are polynomials in $A$. By Lemma~\ref{on}, $A$ is a normal matrix, and by Propositions~\ref{ol} and \ref{ok}, $\{A_0,A_1=A,\ldots,A_D\}\subseteq\A$. Moreover, since $\dim(\A)=D+1$, the set $\{A_0,A_1=A,\ldots,A_D\}$ is a basis of $\A$. The result follows from Theorem~\ref{eD} (Characterization~C).

$(\Ra)$ Assume that $\G=\G(A)$ is a distance-regular digraph. The result follows immediately from Theorem~\ref{eD} (Characterization~C) and Proposition~\ref{Cg}(iii).
\end{proof}


\section{Spectral excess theorem for distance-regular digraphs}
\label{qa}

In this section, we prove spectral excess theorem for distance-regular digraphs. A version of such a theorem can be also found in \cite[Theorem~4.6]{Ogr}

\begin{lemma}
\label{qb}
With reference to {\rm Definition~\ref{of}}, let $\G$ denote a $k$-regular digraph with $d + 1$ distinct eigenvalues, spectrally maximum diameter $D = d$, adjacency algebra $\A$, and predistance polynomials $\{p_0,p_1,\ldots,p_D\}$. Then, $\frac{1}{|X|}\sum_{x\in X} |\G_D^{\ra}(x)|\le p_D(k)$, and equality holds if and only if $A_D=p_D(A)$.
\end{lemma}

\begin{proof} Our proof is within the same lines as the proof of \cite[Lemma~1]{Fsp}, where the authors deal with undirected graph.

Let $\{A_i\}_{i=0}^D$ denote distance-$i$ matrices of $\G$. Consider the vector space $\T=\A+\D$, where $\A=\Span\{A^0,A^1,\ldots,A^d\}$ and $\D=\Span\{A_0,A_1,\ldots,A_D\}$, together with the inner product \eqref{oe}. Note that $\A+\D$ above is the ordinary sum of the vector spaces, not the direct sum. Consider the orthogonal projection $\T\ra\A$, onto the vector space $\A$, denoted by $S\mapsto\proj_{\A}(S)$, and note that $\{p_i(A)\}_{i=0}^D$ is an orthogonal basis for $\A$ (with respect to the inner product \eqref{oe}). Since $\sum_{i=0}^D A_i =J$, $\sum_{i=0}^D p_i(A)=J$ (see Lemma~\ref{oi}), $p_i$ is of degree $i$ $(0\le i\le D)$ and $(A^i)_{xy}=0$ if $\partial(x,y)>i$, the orthogonal projection of $A_D$ $(A_D\ne \O)$ is 
\begin{align}
\proj_{\A}(A_D) &= \sum_{j=0}^D \frac{\langle A_D, p_j(A)\rangle}{\|p_j\|^2} p_j(A)=\frac{\langle A_D, p_D(A)\rangle}{\|p_D\|^2} p_D(A)\nonumber\\
&=\frac{\langle A_D, J\rangle}{\|p_D\|^2} p_D(A)
=\frac{\|A_D\|^2}{p_D(k)} p_D(A)\label{qc}.
\end{align}
Recall that $p_D(k)=\|p_D\|^2>0$ and note
$$
\|\proj_{\A}(A_D)\|^2=\left\langle\frac{\|A_D\|^2}{p_D(k)} p_D(A),\frac{\|A_D\|^2}{p_D(k)} p_D(A)\right\rangle= \frac{\|A_D\|^4}{p_D(k)}.
$$
Now consider the equality $A_D=\proj_{\A}(A_D)+N$, where $N\in\A^\bot$. The Pythagorean theorem and \eqref{qc} yield
$$
0\le \|N\|^2 = \|A_D\|^2 - \|\proj_{\A}(A_D)\|^2=\|A_D\|^2\left(1- \frac{\|A_D\|^2}{{p_D(k)}}\right).
$$
Now we have $\|A_D\|^2=\frac{1}{|X|}\trace(A_D \ol{A}^\top_D)=
\frac{1}{|X|}\sum_{x\in X}\sum_{z\in X} (A_D\circ A_D)_{xz}=\frac{1}{|X|}\sum_{x\in X}|\G^{\ra}_D(x)|$, and the inequality follows. Moreover, the equality holds if and only if $N=\O$, i.e., $A_D=p_D(A)$.
\end{proof}

\begin{lemma}
\label{qd}
Let $A$ denote a normal matrix with spectrum $\spec(A)=\{[\lambda_0]^{m(\lambda_0)},\ldots,[\lambda_d]^{m(\lambda_d)}\}$. With reference to the inner product \eqref{oe},
$$
\langle p,q\rangle=\frac{1}{|X|} \sum_{j=0}^d m(\lambda_j) p(\lambda_j) \ol{q(\lambda_j)}.
$$
\end{lemma}

\begin{proof}
For each eigenvalue $\lambda_i$ $(0\le i\le d)$ let $U_i$ denote the matrix whose columns form an orthonormal basis for the eigenspace $\ker(A-\lambda_i I)$, and note that $\dim(\ker(A-\lambda_i I))=m(\lambda_i)$. Abbreviate $m_i=m(\lambda_i)$ $(0\le i\le d)$. In the proof of Lemma~\ref{on} we show that 
$$
A=\lambda_0E_0+\lambda_1E_1+\cdots+\lambda_d E_d\qquad(\mbox{where }E_i:=U_i\ol{U_i}^{\top}).
$$
Since $E_iE_j=\delta_{ij}E_i$, $\ol{E_i}^\top=E_i$ and $\trace(E_i)=m_i$ $(0\le i\le d)$, for any polynomials $p,q\in\CC_d[t]$ we have
$$
p(A) = p(\lambda_0) E_0 + \cdots + p(\lambda_d) E_d,\qquad
q(A) = q(\lambda_0) E_0 + \cdots + q(\lambda_d) E_d
$$
and
\begin{align*}
\langle p,q\rangle &= \frac{1}{X} \trace(p(A)\ol{q(A)}^\top)\\
&= \frac{1}{X} \trace\big(p(\lambda_0)\ol{q(\lambda_0)}E_0+\cdots+p(\lambda_d)\ol{q(\lambda_d)}E_d\big)\\
&= \frac{1}{|X|} \sum_{j=0}^d m_j p(\lambda_j) \ol{q(\lambda_j)}.
\end{align*}
The result follows.
\end{proof}

\begin{lemma}
\label{qe}
With reference to {\rm Definition~\ref{of}}, let $A$ denote a normal matrix with spectrum $\spec(A)=\{[\lambda_0]^{m(\lambda_0)},\ldots,[\lambda_d]^{m(\lambda_d)}\}$ such that $A\jj=\lambda_0\jj$, and let $\{p_0,p_1,\ldots,p_d\}$ denote a set of predistance polynomials. Set $\pi_i=\prod_{j=0(j\not=i)}^d(\lambda_i-\lambda_j)$ $(0\le i\le d)$. Then 
$$
p_d(\lambda_0)=
\frac{|X|}{|\pi_0|^2(m(\lambda_0))^2}\left(\sum_{j=0}^d \frac{1}{m(\lambda_j)|\pi_j|^2}\right)^{-1}
$$
\end{lemma}

\begin{proof}
Our proof use similar technique that can be fond in \cite{CFfg}. Define  polynomials $Z_i(t)$ $(1\le i\le d)$ of degree $d-1$ in the following way
$$
Z_i(t)=\prod\limits_{\stackrel{j=1}{j\not=i}}^d (t-\lambda_j).
$$
Note that
$$
Z_i(\lambda_\ell)=\left\{
\begin{array}{ll}
\frac{\pi_0}{\lambda_0-\lambda_i}, & \mbox{if } \ell=0,\\
0, & \mbox{if } \ell\ne 0 \mbox{ or }\ell\ne i,\\
\frac{\pi_i}{\lambda_i-\lambda_0}, & \mbox{if } \ell=i
\end{array}
\right.\qquad (1\le i\le d),
$$
which implies
\begin{align*}
0 &= \langle p_d,Z_i \rangle\\
&= \frac{1}{|X|} \sum_{\ell=0}^d m(\lambda_\ell) p_d(\lambda_\ell) \ol{Z_i(\lambda_\ell)}\qquad(\mbox{see Lemma~\ref{qd}})\\
&= \frac{1}{|X|}\left(
m(\lambda_0) p_d(\lambda_0) \ol{Z_i(\lambda_0)}+
m(\lambda_i) p_d(\lambda_i) \ol{Z_i(\lambda_i)}\right)\\
&= \frac{1}{|X|}\left(
m(\lambda_0) p_d(\lambda_0) \frac{\ol{\pi_0}}{\ol{\lambda_0}-\ol{\lambda_i}}+
m(\lambda_i) p_d(\lambda_i) \frac{\ol{\pi_i}}{\ol{\lambda_i}-\ol{\lambda_0}}\right).
\end{align*}
The last line from above yields
\begin{equation}
\label{qf}
m(\lambda_i) p_d(\lambda_i) \ol{\pi_i} = m(\lambda_0) p_d(\lambda_0) \ol{\pi_0}
\qquad (1\le i\le d).
\end{equation}
Taking conjugate of \eqref{qf}, we also have $m(\lambda_i) \ol{p_d(\lambda_i)} {\pi_i} = m(\lambda_0) \ol{p_d(\lambda_0)} {\pi_0}$. 
In the end, by \eqref{qf}, we have
\begin{align*}
p_d(\lambda_0) &= \underbrace{\|p_d\|^2}_{>0} =\langle p_d, p_d\rangle\\
&= \frac{1}{|X|} \sum_{j=0}^d m(\lambda_j) p_d(\lambda_j) \ol{p_d(\lambda_j)}
\qquad(\mbox{see Lemma~\ref{qd}})\\
&= \frac{1}{|X|} m(\lambda_0) p_d(\lambda_0) \underbrace{\ol{p_d(\lambda_0)}}_{=p_d(\lambda_0)}
+\frac{1}{|X|} \sum_{j=1}^d m(\lambda_j) p_d(\lambda_j) {p_d(\lambda_j)}\\
&= \frac{1}{|X|} m(\lambda_0) p_d(\lambda_0) {p_d(\lambda_0)}
+\frac{1}{|X|} \sum_{j=1}^d m(\lambda_j) 
\frac{m(\lambda_0) p_d(\lambda_0) \ol{\pi_0}}{m(\lambda_j)\ol{\pi_j}} 
\frac{m(\lambda_0) \ol{p_d(\lambda_0)} {\pi_0}}{m(\lambda_j){\pi_j}}\\
&= \frac{(p_d(\lambda_0))^2}{|X|} m(\lambda_0)
+\frac{(p_d(\lambda_0))^2}{|X|} |\pi_0|^2 (m(\lambda_0))^2 
\sum_{j=1}^d \frac{1}{m(\lambda_j)|\pi_j|^2}\\
&=\frac{(p_d(\lambda_0))^2}{|X|} |\pi_0|^2 (m(\lambda_0))^2
\sum_{j=0}^d \frac{1}{m(\lambda_j)|\pi_j|^2}.
\end{align*}
The result follows.
\end{proof}

\begin{theorem}[Characterization~J]
\label{qg}
Let $\G$ denote a simple strongly connected digraph with vertex set $X$, adjacency matrix $A$, diameter $D$ and spectrum $\spec(\G)=\{[\lambda_0]^{m(\lambda_0)},\ldots,[\lambda_d]^{m(\lambda_d)}\}$. Then, $\G=\G(A)$ is a distance-regular digraph if and only if $\G$ is regular, has spectrally maximum diameter (i.e., $D=d$), $A^\top$ is a polynomial in $A$, and the following equality holds
\begin{equation}
\label{qk}
\frac{1}{|X|} \sum_{x\in X} \G^{\ra}_D(x)=|X|\left(\sum_{j=0}^d \frac{|\pi_0|^2}{m(\lambda_j)|\pi_j|^2}\right)^{-1},
\end{equation}
where $\pi_i=\prod_{j=0\,(j\not=i)}^d(\lambda_i-\lambda_j)$ $(0\le i\le d)$.
\end{theorem}

\begin{proof}
$(\La)$ Assume that $\G$ is $k$-regular (i.e., $\lambda_0=k$), has spectrally maximum diameter $(D=d)$, $A^\top$ is polynomial in $A$, and that \eqref{qk} holds. Since $\G$ is simple strongly connected digraph, $m(\lambda_0)=1$ (see Subsection~\ref{eb}). By Lemmas~\ref{qb} and \ref{qe}, \eqref{qk} yields $p_D(A)=A_D$. The result follows from Theorem~\ref{om} (Characterization~I).

$(\Ra)$ Routine.
\end{proof}


\section{Further directions}
\label{jD}

\medskip
In the near future, we would like to weaken the assumption of Theorem~\ref{om} and to offer a soultion to  Research problem~\ref{nx}. It is going for showing that the condition  $A^\top\in\{A_0,A_1=A,\ldots,A_D\}$ together with $A_D$ is polynomial in $A$ is equivalent with the property that both $A^\top$ and $A_D$ are polynomial in $A$ (see Theorem~\ref{om} (Characterization~I)).

\begin{researchProblem}
\label{nx}
Let $\G$ denote a simple strongly connected digraph with vertex set $X$, diameter $D$ and let $A_i$ $(0\le i\le D)$ denote the distance-$i$ matrices. Prove or disprove the following claim: $\G$ is a distance-regular digraph if and only if $\G$ is regular, has spectrally maximum diameter $(D=d)$, $A^\top\in\{A_0,A_1=A,\ldots,A_D\}$ and distance-$D$ matrix $A_D$ is a polynomial in $A$.
\end{researchProblem}

\section*{Acknowledgments}

This work is supported in part by the Slovenian Research Agency (research program P1-0285 and research projects J1-3001 and N1-0353).

\section*{Declaration of competing interest}

The authors declare that they have no known competing financial interests or personal relationships that could have appeared to influence the work reported in this paper.


{\small
\bibliographystyle{references}
\bibliography{directedDRG}
}



\end{document}